\documentclass[a4paper,11pt]{article}
\usepackage{index}
\usepackage[utf8]{inputenc}
\usepackage{longtable}
\usepackage{vmargin,hhline,pifont}
\usepackage{amsmath}
\usepackage{amsthm}
\usepackage{amsfonts}
\usepackage{amssymb}

\newtheorem{theorem}{Theorem}[section]
\newtheorem{proposition}[theorem]{Proposition}
\newtheorem{lemma}[theorem]{Lemma}

\theoremstyle{definition}
\newtheorem{definition}[theorem]{Definition}
\newtheorem{example}[theorem]{Example}

\newtheorem{corollary}[theorem]{Corollary}
\theoremstyle{remark}
\newtheorem{remark}[theorem]{Remark}

\newcommand{\ndash}{\nobreakdash-\hspace{0pt}}

\DeclareMathOperator{\Ad}{Ad}
\DeclareMathOperator{\ad}{ad}
\DeclareMathOperator{\Id}{Id}

\begin{document}

\title{Orthogonal symmetric affine Kac-Moody algebras}
\author{Walter Freyn}

\maketitle

\begin{abstract}
Riemannian symmetric spaces are fundamental objects in finite dimensional differential geometry. An important problem is the construction of symmetric spaces for generalizations of simple Lie groups, especially their closest infinite dimensional analogues known as Kac-Moody groups. We solve this problem and construct affine Kac-Moody symmetric spaces in a series of several papers. This paper focuses on the algebraic side: More precisely, we introduce OSAKAs, the algebraic structures used to describe the connection between affine Kac-Moody symmetric spaces and affine Kac-Moody algebras and describe their classification. 
\end{abstract}

\section{Introduction}

Riemannian symmetric spaces are fundamental objects in finite dimensional differential geometry, displaying numerous connections with Lie theory, physics, and analysis. The search for infinite dimensional symmetric spaces associated to affine Kac-Moody algebras has been an open question for 20 years, since it was first asked by C.-L. Terng in~\cite{Terng95}. We present a complete solution to this problem in a series of several papers, dealing successively with the functional analytic, the algebraic and the geometric aspects. In this paper we introduce and classify \emph{o}rthogonal \emph{s}ymmetric \emph{a}ffine \emph{Ka}c-Moody algebras (OSAKAs); OSAKAs are the central objects in the classification of affine Kac-Moody symmetric spaces as they provide the crucial link between the geometric and the algebraic side of the theory.  The functional analytic fundamentals and the geometry of Kac-Moody symmetric spaces are described elsewhere - see~\cite{Freyn07, Freyn10a, Freyn12c, Freyn12d, Freyn12e, Freyn12f}.

The so-called Kac-Moody algebras were introduced and first studied in the 60's independently by V.\ G.\ Kac~\cite{Kac68}, R.\ V.\ Moody~\cite{Moody69}, I.\ L.\ Kantor~\cite{Kantor68} and D.-N.\ Verma (unpublished) as a broad generalization of semisimple Lie algebras. There are distinguished subclasses of independent interest: for example finite dimensional simple Lie algebras are Kac-Moody algebras of so-called \emph{spherical type}. The most studied infinite dimensional subclass is the one of \emph{affine Kac-Moody algebras} and their associated \emph{affine Kac-Moody groups}. These affine Kac-Moody groups can be constructed as certain torus extensions $\widehat{L}(G, \sigma)$ of (possibly twisted) loop groups $L(G, \sigma)$. Correspondingly, affine Kac-Moody algebras are $2$\ndash dimensional extensions of (possibly twisted) loop algebras $L(\mathfrak{g}, \sigma)$. In this notation $G$ denotes a compact or complex simple Lie group, $\mathfrak{g}$ its Lie algebra and $\sigma\in \textrm{Aut}(G)$ a diagram automorphism, defining the ``twist'' and $L(\mathfrak{g}, \sigma)$ is defined to consists of functions $f:\mathbb{R}\rightarrow \mathfrak{g}$ such that $f(t+2\pi)=\sigma f(t)$, satisfying certain \emph{regularity conditions}. Depending on the regularity assumptions on the loops (e.g.~of Sobolev class $H^k$), one constructs families of different completions of the minimal (=algebraic) affine Kac-Moody groups (resp.~Kac-Moody algebras), that is those which consists of polynomial loops. Following Jacques Tits, completions defined by imposing functional analytic regularity conditions on the loops are called \emph{analytic completions} in contrast to the more algebraic \emph{formal completion}. These analytic completions play an important role in certain branches of mathematics and physics, integrable systems and differential geometry (for references see~\cite{Freyn10a, Freyn12c}).

In this paper we introduce and classify OSAKAs. Roughly speaking, OSAKAs consist of pairs of Kac-Moody algebras together with some involution, whose fixed point algebra is a loop algebra of \emph{``compact type''}. 
Their construction generalizes the notion of OSLAs (\emph{o}rthogonal \emph{s}ymmetric \emph{L}ie \emph{a}lgebras) well known from the theory of finite dimensional Riemannian symmetric spaces, as there is a bijective correspondence between OSLAs and simply connected Riemannian symmetric spaces (see~\cite{Helgason01}, chapter V).
A similar correspondence exists between OSAKAs and affine Kac-Moody symmetric spaces. In this paper we focus entirely on the algebraic side of the theory, studying OSAKAs. The relation between OSAKAs and affine Kac-Moody symmetric spaces will be investigated elsewhere.

 In section~\ref{sect:algebraic_kac_moody} we review some basics of the theory of (algebraic) affine Kac-Moody algebras. In case of ambiguity, we use the prefix ``algebraic'' for what is called in the literature affine Kac-Moody algebras, to distinguish this class from the wider class of \emph{geometric affine Kac-Moody algebras}. We introduce this new class of Lie algebras in section~\ref{sect:geometric_kac_moody}, constructing its members in a geometric way as certain extensions of Lie algebras of maps. 
  Geometric affine Kac-Moody algebras contain the subclass of affine Kac-Moody algebras (and their analytic completions) as a special case; nevertheless there are important examples, which are not contained in this subclass. 
  Then  in section~\ref{sect:irred_geometric_kac_moody} we investigate irreducible geometric affine Kac-Moody algebras; the description of irreducible geometric affine Kac-Moody algebras is central to the classification of OSAKAs.
  In section~\ref{sect:real_forms} we investigate real forms of geometric affine Kac-Moody algebras.  These real forms will be used in section~\ref{sect:osaka} to construct OSAKAs. In section~\ref{sect:compact_loop} we describe \emph{compact loop algebras}; they appear in the classification of OSAKAs as the fixed point algebras of the involutions. Then we describe and prove the classification theorem, relating the classification of OSAKAs to the classification of affine Kac-Moody algebras and their involutions. These in turn are classified by previous work (see~\cite{Heintze08, Heintze09, Kac90, BBMR95, MessaoudRousseau03, Levstein88}). In the final section~\ref{sect:a_1_example} we describe the orthogonal symmetric affine Kac-Moody algebras associated to $A_1^{(1)}$ in detail. We will describe elsewhere the relation between OSAKAs and affine Kac-Moody symmetric spaces. 
 My special thanks go to the anonymous referee for an extensive list of suggestions leading to significant improvement of the article.

\section{Algebraic affine Kac-Moody algebras}
\label{sect:algebraic_kac_moody}

In this section we review the algebraic theory of Kac-Moody algebras with an emphasis towards the class of affine Kac-Moody algebras.  References are the books~\cite{Kac90, Moody95, Carter05}. 

The basic idea of Kac-Moody theory follows from Cartan's classification of complex semisimple Lie algebras: this classification associates to each semisimple complex Lie algebra a certain matrix, its so-called \emph{Cartan matrix}. Conversely the Lie algebra can be recovered from its Cartan matrix via an explicit construction by generators and relations.  It is then a natural question to investigate the algebras, that result from more general types of matrices than the classical Cartan matrices by the same constructions:  the most important class of algebras constructed in this way is the one of Kac-Moody algebras.

We start thus with the definition of a Cartan matrix:

\begin{definition}

A \emph{Cartan matrix} $A^{n\times n}$ is a square matrix with integer
coefficients such that
\begin{enumerate}
  \item $a_{ii}=2$ and $a_{i\not=j}\leq 0 \label{geometriccondition}$,
  \item $a_{ij}=0 \Leftrightarrow a_{ji}=0\label{liealgebracondition} $,
  \item There is a vector $v>0$ (component wise) such that $Av>0$ (component wise).
\end{enumerate}
\end{definition}

A Cartan matrix $A^{n\times n}$ is called \emph{decomposable} iff $\{1, 2, \dots, n\}$ has a decomposition into two non-empty sets $N_1$ and $N_2$ such that $a_{ij}=0$ for $i\in N_1$ and $j\in N_2$. Otherwise it is called \emph{indecomposable}.

\begin{example}

There are -- up to equivalence -- four different $2$-dimensional Cartan matrices:
$$\left(\begin{array}{rr} 2&0\\0&2  \end{array}\right),
\left(\begin{array}{rr} 2&-1\\-1&2  \end{array}\right),
\left(\begin{array}{rr} 2&-1\\-2&2  \end{array}\right),
\left(\begin{array}{rr} 2&-1\\-3&2  \end{array}\right).$$

\noindent The first Cartan matrix is decomposable, the other three Cartan matrices are indecomposable. These Cartan matrices correspond to the finite dimensional simple Lie algebras
$\mathfrak{a}_{1}\times \mathfrak{a}_1, \mathfrak{a}_{2}, \mathfrak{b}_{2}, \mathfrak{g}_{2}$ respectively. Their dimensions are $6$, $8$, $10$, and $14$.  
Over the field $\mathbb{C}$ we have the equivalences: 
\begin{align*}
\mathfrak{g}(\mathfrak{a}_{1}\times \mathfrak{a}_{1})&\cong\mathfrak{sl}(2,\mathbb{C})\times \mathfrak{sl}(2, \mathbb{C}), \\
\mathfrak{g}(\mathfrak{a}_{2})&\cong \mathfrak{sl}(3,\mathbb{C}),\\
\mathfrak{g}(\mathfrak{b}_{2})&\cong\mathfrak{so}(5,\mathbb{C}),\\
\mathfrak{g}(\mathfrak{g}_{2})&\cong\mathfrak{g}_2(\mathbb{C}).
\end{align*}
\end{example}

\noindent A complete list of indecomposable Cartan matrices consists of the following
\begin{displaymath}
\mathfrak{a}_{n}, \mathfrak{b}_{n, n\geq 2}, \mathfrak{c}_{n, n\geq 3}, \mathfrak{d}_{n, n \geq 4}, \mathfrak{e}_{6}, \mathfrak{e}_{7}, \mathfrak{e}_{8}, \mathfrak{f}_{4}, \mathfrak{g}_{2}\,. 
\end{displaymath}

\noindent They correspond to the simple Lie algebras of the same name (see~\cite{Carter05}, sect. 8).

\begin{definition}
An Cartan matrix $A^{n\times n}$ of \emph{affine type} is a square matrix with integer
coefficients, such that
\begin{enumerate}
  \item $a_{ii}=2$ and $a_{i\not=j}\leq 0 \label{geometricconditionaffine}$.
  \item $a_{ij}=0 \Leftrightarrow a_{ji}=0\label{liealgebraconditionaffine}$.
  \item There is a vector $v>0$ (component wise) such that $Av=0$.
\end{enumerate}
\end{definition}

\begin{example}
There are -- up to equivalence -- two different $2$-dimensional affine Cartan matrices:
$$\left(\begin{array}{rr} 2&-2\\-2&2  \end{array}\right),
\left(\begin{array}{rr} 2&-1\\-4&2  \end{array}\right)\ .$$

\noindent They correspond to the non-twisted affine Kac-Moody algebra
$\widetilde{\mathfrak{a}}_{1}$ and the twisted affine Kac-Moody algebra $\widetilde{\mathfrak{a}}_{1}'$ respectively. Both algebras are infinite dimensional.
\end{example}

The distinction between \emph{twisted} and \emph{non-twisted} affine Kac-Moody algebras is most explicit from the construction of the Kac-Moody algebras as an extension of a loop algebra (see section~\ref{sect:geometric_kac_moody}). It is also reflected structurally in the dimensions of imaginary root spaces (see~\cite{Carter05}, section 18.2 and 18.4).

\begin{enumerate}
\item The indecomposable non-twisted affine Cartan matrices are
\begin{displaymath}
\widetilde{\mathfrak{a}}_{n}, \widetilde{\mathfrak{b}}_{n}, \widetilde{\mathfrak{c}}_{n}, \widetilde{\mathfrak{d}}_n, \widetilde{\mathfrak{e}}_{6}, \widetilde{\mathfrak{e}}_{7}, \widetilde{\mathfrak{e}}_{8}, \widetilde{\mathfrak{f}}_{4}, \widetilde{\mathfrak{g}}_{2}\,.
\end{displaymath}
They correspond to affine Kac-Moody algebras of the same name. In fact every non-twisted affine Cartan matrix $\widetilde{X}_l$ can be reduced to the corresponding Cartan matrix $X_l$ by the removal of the first column and the first line. Hence there is a bijection between indecomposable Cartan matrices and indecomposable non-twisted affine Cartan matrices. This bijection of Cartan matrices has its counterpart in a bijection between simple complex Lie algebras and  non-twisted affine complex Kac-Moody algebras, which is made explicit via the loop algebra realizations; see section~\ref{sect:geometric_kac_moody}. 

\item The indecomposable twisted affine Cartan matrices are
\begin{displaymath} 
\widetilde{\mathfrak{a}}_{1}', \widetilde{\mathfrak{c}}_{l}', \widetilde{\mathfrak{b}}_l^t, \widetilde{\mathfrak{c}}_{l}^t, \widetilde{\mathfrak{f}}_{4}^t, \widetilde{\mathfrak{g}}_{2}^t\,.
\end{displaymath}

Here, as in the non-twisted case, the name of the diagram $(\widetilde{X})$ describes the root system of the associated Kac-Moody algebras and the superscript $t$ labels it as a twisted one.
\end{enumerate}

They correspond to Dynkin diagrams of the same name. The Kac-Moody algebras associated to those diagrams can be constructed as fixed point algebras of the so-called twisted diagram automorphisms $\sigma$ of some non-twisted Kac-Moody algebra $X$. This construction suggests an alternative notation describing a twisted Kac-Moody algebra by the order of $\sigma$ and the type of $X$ (for details see~\cite{Carter05}, p.~451).

\noindent To all those classes of Cartan matrices one can associate Lie algebras, called their algebraic Lie algebra realizations (see~\cite{Kac90}, 1.1):

\begin{definition}
Let $A^{n \times n}$ be a finite or affine Cartan matrix. The \emph{realization} of $A$ is defined as the triple $(\mathfrak{h}, H, H^{v})$ where $\mathfrak{h}$ denotes a complex vector space, $H=\{\alpha_1,\dots \alpha_n\}\in \mathfrak{h}^v$ and $H^{v}:=\{\alpha_1^v, \dots \alpha_n^v\}\in \mathfrak{h}$ such that 
\begin{itemize}
\item $H$ and $H^v$ are linearly independent and satisfy,
\item $\langle h_i^v, h_i\rangle=a_{ij}$ for $i=1, \dots, n$,
\item $\textrm{dim}\mathfrak{h}-n=n-l$ where $l$ denotes the rank of $A$.
\end{itemize}
\end{definition}

\begin{remark}
An affine $n\times n$-Cartan matrix has rank $l=n-1$. Hence we get $\textrm{dim}(\mathfrak{h})=n+1$.
\end{remark}

From a realization we construct the Lie algebra realization as follows:

\begin{definition}

Let $A^{n \times n}$ be a generalized Cartan matrix. The Lie algebra realization of $A$, denoted $\mathfrak{g}(A)$, is the Lie algebra
$$\mathfrak{g}(A^{n\times n}) = \langle \mathfrak{h}, e_i, f_i,  i=1,\dots, n| \textrm{R}_1, \dots, \textrm{R}_6\rangle\,,$$ where
\begin{displaymath}
  \begin{array}{cl}
    \textrm{R}_1:& [h_i,h_j]=0\ \textrm{for}\ g_i, h_i\in \mathfrak{h}\,,\\
    \textrm{R}_2:& [e_i,f_j]=\alpha_i^v \delta_{ij}\,,\\
    \textrm{R}_3:& [h, e_j]=\langle \alpha_i, h\rangle e_j\,,\\
    \textrm{R}_4:& [h, f_j]=-\langle \alpha_i, h\rangle f_j\,,\\
    \textrm{R}_5:& (\textrm{ad} e_i)^{1-a_{ji}}(e_j)=0 \hspace{3pt}\quad \textrm{for}\ i\not= j\,,\\
    \textrm{R}_6:& (\textrm{ad} f_i)^{1-a_{ji}}(f_j)=0\hspace{3pt}\quad\textrm{for}\ i\not= j\,.\\
  \end{array}
\end{displaymath}
\end{definition}

\noindent Here relation $\textrm{R}_1$ describes that the elements $h_i, i=1, \dots n$ generate some Abelian subalgebra. The relations $\textrm{R}_2$ and $\textrm{R}_3$ show, that for each single $i$ the triple $e_i, f_i, h_i$ generates a $3$-dimensional subalgebra isomorphic to $\mathfrak{sl}(2)$. Relations $\textrm{R}_5$ and $\textrm{R}_6$ are called the \emph{Serre relations}. They describe the generation of additional root spaces by Lie brackets of generators $e_i$ (resp.~$f_i$), $i\in 1,\dots, n$.

If a (generalized) Cartan matrix $A^{n+m\times n+m}$ is decomposable into the direct sum of two Cartan matrices $A^{n\times n}$ and $A^{m\times m}$ then a similar decomposition holds for the realizations:
\begin{displaymath}
\mathfrak{g}(A^{n+m\times n+m})=\mathfrak{g}(A^{n\times n})\oplus\mathfrak{g}(A^{m\times m})\,.
\end{displaymath}

We introduce some further terminology. Let $\mathcal{G}$ be an affine Kac-Moody algebra. $\mathcal{G}$ has a \emph{triangular decomposition} 
$\mathcal{G}=\mathcal{N}^-\oplus \mathcal{H}\oplus \mathcal{N}^+\,.$
where $\mathcal{H}$ (resp. $\mathcal{N}^+$, $\mathcal{N}^-$) are generated by the $h_i$ (resp. $e_i$, $f_i$).
A \emph{Cartan subalgebra} $\mathcal{H}$ of $\mathcal{G}$ is a maximal $\ad(\mathcal{G})$-diagonalizable subalgebra. It is a consequence of the results in~\cite{PetersonKac83} that all Cartan subalgebras in an affine Kac-Moody algebras are conjugate.
$\mathcal{B}^{\pm}:=\mathcal{H}\oplus \mathcal{N}^{\pm}$ is the standard positive (negative) Borel subalgebra. A \emph{Borel subalgebra} $\mathcal{B}$ is subalgebra of $\mathcal{G}$ which is conjugate to $\mathcal{B}^{+}$ or $\mathcal{B}^{-}$. $\mathcal{B}^{+}$ and $\mathcal{B}^{-}$ are not conjugate.
Hence the set of all Borel subalgebras consists of exactly two conjugacy classes, called the \emph{positive} and the \emph{negative} conjugacy class.

\begin{definition}
Let $\mathfrak{g}_{\mathbb{C}}$ be a complex Lie algebra. A \emph{real form} of $\mathfrak{g}_{\mathbb{C}}$ is a Lie algebra $\mathfrak{g}_{\mathbb{R}}$ 
$$\mathfrak{g}_{\mathbb{C}}=\mathfrak{g}_{\mathbb{R}}\otimes \mathbb{C}\, .$$
\end{definition}

It is well-known that real forms are in bijection with conjugate-linear involution: Let $\mathfrak{g}_{\mathbb{R}}$ be a real form of $\mathfrak{g}_{\mathbb{C}}$, then conjugation along $\mathfrak{g}_{\mathbb{R}}$ is a conjugate linear involution of $\mathfrak{g}_{\mathbb{C}}$. Conversely the eigenvalues of some involution are $\pm 1$. Hence a (conjugate-linear) involution $\overline{\rho}:\mathfrak{g}_{\mathbb{C}}\longrightarrow \mathfrak{g}_{\mathbb{C}}$ induces a splitting of $\mathfrak{g}_{\mathbb{C}}$ into the $\pm 1$-eigenspaces  $\mathfrak{g}_{\pm}$. If $x\in \mathfrak{g}_{\pm}$ then $ix\in \mathfrak{g}_{\mp}$. Consequently $\mathfrak{g}_{-1}=i \mathfrak{g}_{1}$ and 

$$\mathfrak{g}_{\mathbb{C}}=\mathfrak{g}_{1}\oplus \mathfrak{g}_{-1}=\mathfrak{g}_{1}\oplus i \mathfrak{g}_{1}=\mathfrak{g}_{1}\otimes \mathbb{C}\,.$$ 

Involutions of affine Kac-Moody algebras (and consequently also the real forms) fall into two types, depending on their behaviors on the two conjugacy classes of Borel subalgebras (see~\cite{MessaoudRousseau03} and \cite{Heintze09}, section 6).
\begin{itemize}
\item {\bf First kind:} They preserve the  two conjugacy classes of Borel subalgebras.
\item {\bf Second kind:} They exchange the two conjugacy classes of Borel subalgebras.
\end{itemize}


\section{Geometric affine Kac-Moody algebras}
\label{sect:geometric_kac_moody}

The loop algebra construction of (algebraic) affine Kac-Moody algebras is developed in the books \cite{Kac90, Carter05} from an algebraic point of view. We follow the more geometric approach to the construction of affine Kac-Moody algebras as extensions of loop algebras as described in~\cite{Heintze09}. Remark that the class of ``geometric affine Kac-Moody algebras'' we construct is far more general. In fact we will see that affine Kac-Moody algebras correspond to irreducible geometric affine Kac-Moody algebras.

\noindent Let $\mathfrak{g}$ be a finite dimensional reductive Lie algebra over $\mathbb{F}=\mathbb{R}$ or $\mathbb{C}$. Hence by definition $\mathfrak{g}=\mathfrak{g}_{s}\oplus \mathfrak{g}_{a}$ is a direct product of a semisimple Lie algebra $\mathfrak{g}_s$ with an Abelian Lie algebra $\mathfrak{g}_{a}$. Let furthermore $\sigma=(\sigma_{s}, \sigma_{a})$ be some involution of $\mathfrak{g}$, such that $\sigma_s \in \textrm{Aut}(\mathfrak{g}_{s})$ denotes an automorphism of finite order of $\mathfrak{g}_s$ which preserves each simple factor $\mathfrak{g}_{s}$ and $\sigma_a$ is the identity on $\mathfrak{g}_{a}$. If $\mathfrak{g}_{s}$ is a Lie algebra over $\mathbb{R}$ we assume $\mathfrak{g}_{s}$ to be a Lie algebra of compact type. We define the loop algebra $L(\mathfrak g, \sigma)$ as follows
\begin{displaymath}
L(\mathfrak g, \sigma):=\{f:\mathbb{R}\longrightarrow \mathfrak{g}\hspace{3pt}|f(t+2\pi)=\sigma f(t), f \textrm{ satisfies some regularity conditions}\}\label{abstractkacmoodyalgebra}\,.
\end{displaymath}

We use the notation $L(\mathfrak g, \sigma)$ to describe in a unified way constructions that can be realized explicitly for loop algebras satisfying various regularity conditions. Regularity conditions used in applications include the following:
\begin{itemize}
\item $H^0$-Sobolev loops, denoted $L^0\mathfrak{g}$,
\item smooth, denoted $L^{\infty}\mathfrak{g}$,
\item real analytic, denoted $L_{\textrm{an}}\mathfrak{g}$,
\item (after complexification of the domain of definition)  holomorphic on $\mathbb{C}^*$, denoted $M\mathfrak{g}$,
\item holomorphic on an annulus $A_n\subset \mathbb{C}$, denoted $A_n\mathfrak{g}$, or 
\item algebraic (or equivalently: with a finite Fourier expansion), denoted $L_{alg}\mathfrak{g}$.
\end{itemize}

\begin{definition}
\label{def:irreducible_geometric_KMA}
A (geometric affine) Kac-Moody algebra $\mathfrak{g}$ is called \emph{irreducible} if there is no ideal $\mathfrak{h}\subset \mathfrak{g}$ which is isomorphic to a Kac-Moody algebra.
\end{definition}

\begin{definition}[indecomposable geometric affine Kac-Moody algebra]
\label{def:indecomposable_geometric_KMA}
A geometric affine Kac-Moody algebra $\mathfrak{g}$ is called indecomposable if it does not admit a decomposition $\mathfrak{g}=\mathfrak{g}_1\oplus \mathfrak{g}_2$ into two ideals $\mathfrak{g}_1$ and $\mathfrak{g}_2$. Conversely a \emph{decomposable geometric affine Kac-Moody algebra} is the direct sum of indecomposable geometric affine Kac-Moody algebras.
\end{definition}

\begin{example}
\label{geometricaffinekacmoodyalgebra}
The geometric affine Kac-Moody algebra associated to a pair $(\mathfrak{g}, \sigma)$ as described above is the algebra:
$$\widehat{L}(\mathfrak g, \sigma):=L(\mathfrak g, \sigma) \oplus \mathbb{F}c \oplus \mathbb{F}d\,,$$
equipped with the Lie bracket defined by:
\begin{alignat*}{1}
  [d,f(t)]&:=f'(t)\, ,\\
  [c,c]=[d,d]&:=0\, ,\\
	[c,d]=[c,f(t)]&:=0\, ,\\
  [f,g](t)&:=[f(t),g(t)]_{0} + \omega\left(f(t),g(t)\right)c\,.
\end{alignat*}
Here $f\in L(\mathfrak g, \sigma)$ and $\omega$ is a certain antisymmetric $2$-form on $L(\mathfrak {g}, \sigma)$, satisfying the cocycle condition. If the regularity of $L(\mathfrak g, \sigma)$ is chosen such that $L(\mathfrak g, \sigma)$ contains non-differentiable functions, then $d$ is defined on the (dense) subspace of differentiable functions. This algebra is indecomposable.
\end{example}

The form $\omega$ we may define as follows:
\begin{enumerate}
\item For holomorphic or algebraic loops we may define \begin{displaymath}\omega(f,g):=\textrm{Res}(\langle f, g' \rangle)\,.\end{displaymath}
\item For integrable loops we may use \begin{displaymath}\omega(f,g)=\frac{1}{2\pi}\int_{0}^{2\pi}\langle f, g'\rangle dt\,.\end{displaymath}
\end{enumerate}
By the formula for residua of holomorphic functions (see~\cite{Berenstein91}, section I.10) those two definitions of $\omega$ coincide for polynomial functions and for holomorphic functions. All functional analytic completions which we study contain the space of polynomials as a dense subspace. Hence we can use both formulations equivalently.

Let us reformulate the definition for non twisted affine Kac-Moody algebras in terms of functions on $\mathbb{C}^*$: Suppose $\sigma=\Id$. First we develop a function $f\in L(\mathfrak{g}, \Id)$ into its Fourier series $f(t)=\sum a_n{e}^{int}$. Then this function is naturally defined on a circle $S^1$; we understand this circle to be embedded as the  unit circle $\{z\in \mathbb{C}^*||z|=1\}\subset\mathbb{C}^*$; in this way the parameter $t$ gets replaced by the complex parameter $z:=e^{it}$, with $|z|=1$; understanding the Fourier expansion now as a Laurent expansion of $f$, we can calculate the annulus, $A_n$, on which this series is defined. For example the holomorphic realization $M\mathfrak{g}$, is defined by the condition, that for any $f\in M\mathfrak{g}$ the Fourier expansion describes a Laurent series expansion of a holomorphic function on $\mathbb{C}^*$.

\begin{definition}
The complex geometric affine Kac-Moody algebra associated to a pair $(\mathfrak{g}, \sigma)$ is the algebra
$$\widehat{M\mathfrak g}:=M\mathfrak g \oplus \mathbb{C}c \oplus \mathbb{C}d\,,$$
equipped with the Lie bracket defined by:
\begin{alignat*}{1}
  [d,f(z)]&:=izf'(z)\,,\\
  [c,c]=[d,d]&:=0\,,\\
  [c,d]=[c,f(z)]&:=0\,,\\
  [f,g](z)&:=[f(z),g(z)]_{0} + \omega(f(z),g(z))c\,.
\end{alignat*}
\end{definition}

\noindent As $\frac{d}{dt}e^{itn}=ine^{itn}=inz^n=iz\frac{d}{dz}z^n$ both definitions coincide.

\begin{definition}
A geometric affine Kac-Moody algebra $\widehat{L}(\mathfrak{g}, \sigma)$ is called 
semisimple (resp. simple) if $\mathfrak{g}$ is semisimple (resp. simple).
\end{definition}

\begin{definition}
The \emph{derived geometric affine Kac-Moody algebra} associated to a pair $(\mathfrak{g}, \sigma)$ is the algebra:
\begin{displaymath}\widetilde{L}(\mathfrak g, \sigma):=L(\mathfrak g, \sigma) \oplus \mathbb{F}c\,,\end{displaymath}
with the Lie bracket induced from the affine geometric Kac-Moody algebra.
\end{definition}

This algebra is called ``derived algebra'' as for semisimple geometric affine Kac-Moody algebras it satisfies
\begin{displaymath}\widetilde{L}(\mathfrak g, \sigma)\hspace{3pt}=\hspace{3pt}\left[\widehat{L}(\mathfrak g, \sigma), \widehat{L}(\mathfrak g, \sigma)\right]\,.\end{displaymath}

If $\mathfrak{g}$ is a simple Lie algebra then the associated algebraic and geometric Kac-Moody algebras coincide up to completion:

\begin{lemma}
Suppose $\mathfrak {g}$ is simple.
The realization of $\widehat{L}(\mathfrak g, \sigma)$ with algebraic loops $\widehat{L_{alg}\mathfrak{g}}^{\sigma}$ is a simple, algebraic affine Kac-Moody algebra.
\end{lemma}

This lemma is a restatement of the loop algebra realization of affine Kac-Moody algebras described in~\cite{Carter05}.
If one chooses $\widehat{L}(\mathfrak{g}, \sigma)$ to be a realization of a more general regularity class of loops than algebraic loops then $\widehat{L}(\mathfrak g, \sigma)$ can be described as a completion of $\widehat{L_{alg}\mathfrak{g}}^{\sigma}$, the realization of $\widehat{L}(\mathfrak{g}, \sigma)$ with algebraic loops, with respect to some seminorms or some set of semi-norms. 

Following \cite{PressleySegal86}, p.168 we define a Lie algebra representation $R$ of a Lie algebra $\mathfrak{g}$ to be essentially equivalent to $R'$ if there is an injective map $\varphi: R\longrightarrow R'$ with dense image, equivariant with respect to $\mathfrak{g}$.

\begin{lemma}
The adjoint representation of $\widehat{L_{alg}\mathfrak{g}}^{\sigma}$ on 
$\widehat{L_{alg}\mathfrak{g}}^{\sigma}$ is essentially equivalent to the adjoint representation of $\widehat{L_{alg}\mathfrak{g}}^{\sigma}$ on $\widehat{L}(\mathfrak{g}, \sigma)$ for all functional analytic regularity conditions.
\end{lemma} 

\begin{proof}
As $\widehat{L}(\mathfrak{g}, \sigma)$ is constructed as a completion of $\widehat{L_{alg}\mathfrak{g}}^{\sigma}$, the embedding $\varphi:\widehat{L_{alg}\mathfrak{g}}^{\sigma}\longrightarrow \widehat{L}(\mathfrak{g}, \sigma)$ has a dense image. Furthermore it is injective and $\widehat{L_{alg}\mathfrak{g}}^{\sigma}$-equivariant.
\end{proof}

\begin{remark}
Remark that \emph{formal completions} of algebraic Kac-Moody algebras as described for example in~\cite{Kumar02} do not preserve all symmetries of the Kac-Moody algebra.
\end{remark}

\section{Irreducible geometric affine Kac-Moody algebras}
\label{sect:irred_geometric_kac_moody}

We will now describe the irreducible geometric affine Kac-Moody algebras and investigate, how geometric affine Kac-Moody algebras are constructed from these building blocks.

We start with the investigation of the loop algebra part.

\begin{lemma}
Let $\mathfrak{g}=\mathfrak{g}_a\oplus \mathfrak{g}_s$ be a reductive Lie algebra and $\sigma=\sigma_a\otimes\sigma_s$ an involution. Then $L(\mathfrak{g}_a, \sigma_a)$ and $L(\mathfrak{g}_s, \sigma_s)$ are ideals in $L(\mathfrak{g}_s, \sigma_s)$.
\end{lemma}

\begin{proof}
Decompose any function $f\in L(\mathfrak{g}, \sigma)$ into the component functions $f_a\in L(\mathfrak{g}_a, \sigma_a)$ and $f_s\in L(\mathfrak{g}_s, \sigma_s)$.
Using that $\mathfrak{g}_a$ and $\mathfrak{g}_s$ are ideals in $\mathfrak{g}$, and using that the bracket is defined pointwise we get $[f_s, g_a](t)=[f_s(t),g_a(t)]=0$, which includes the statement.
\end{proof}

\begin{lemma}
Let $\mathfrak{g}_s$ be semisimple and suppose $(\mathfrak{g}_s, \sigma):=\bigoplus_i (\mathfrak{g_i}, \sigma_i)$.
Then \begin{displaymath}L(\mathfrak{g}_s, \sigma)=\bigoplus_i L(\mathfrak{g}_i, \sigma_i)\,.\end{displaymath}
Each algebra $L(\mathfrak{g}_i, \sigma_i)$ is an ideal in $L(\mathfrak{g}_s, \sigma)$.
\end{lemma}

\begin{proof}
Decompose any function $f\in L(\mathfrak{g}_s, \sigma)$ into its component loops $f_i\in L(\mathfrak{g}_i, \sigma_i)$. As the bracket is defined pointwise we get $[f_i, g_j](t)=[f_{i}(t), g_{j}(t)]$ and in consequence that each $L(\mathfrak{g}_i, \sigma_i)$ is an ideal in $L(\mathfrak{g}, \sigma)$. This yields the direct product decomposition.\qedhere
\end{proof}

\noindent Nevertheless, the extension by the derivative and the central element behaves differently:

\begin{lemma}
Assume $\mathfrak{g}$ is a semisimple, non-simple Lie algebra. Then
$$\widehat{L}(\mathfrak{g}, \sigma)\not= \bigoplus_i \widehat{L}(\mathfrak{g}_i, \sigma_i)\,.$$
\end{lemma}

\begin{proof}
The center of $\widehat{L}(\mathfrak{g}, \sigma)$ is $1$-dimensional. In contrast the dimension of the center of $\bigoplus_i \widehat{L}(\mathfrak{g}_i, \sigma_i)$ is equivalent to the number of simple factors of $\mathfrak{g}$.
\end{proof}

\begin{corollary}
Geometric affine Kac-Moody algebras do not decompose into a direct sum of indecomposable geometric affine Kac-Moody algebras. 
\end{corollary}

\begin{lemma}[Ideals of geometric Kac-Moody algebras]
\label{lemma:splitting}

Let as above $\mathfrak{g}:=\mathfrak{g}_a\oplus\bigoplus_i \mathfrak{g}_i$ be the decomposition of a reductive Lie algebra into its Abelian factor $\mathfrak{g}_a$ and its simple factors $\mathfrak{g}_i$. Denote by~$\sigma_i$ (resp.~$\sigma_a$) the restriction of $\sigma$ to $\mathfrak{g}_i$ (resp.~$\mathfrak{g}_a$). Let $\widehat{L}(\mathfrak{g}, \sigma)$ be the associated geometric affine Kac-Moody algebra. Let $\mathbb{F}\in\{\mathbb{R}, \mathbb{C}\}$ depending if the Lie algebra $\mathfrak{g}$ is real or complex. Then the following holds: 
\begin{enumerate}
\item $\widetilde{L}(\mathfrak{g}, \sigma)$ is an ideal in $\widehat{L}(\mathfrak{g}, \sigma)$.
\item $\widetilde{L}(\mathfrak{g}_a, \sigma_a)$ and $\widetilde{L}(\mathfrak{g}_i, \sigma_i)$ are ideals in $\widetilde{L}(\mathfrak{g}, \sigma)$ and in $\widehat{L}(\mathfrak{g}, \sigma)$.
\item Let $\mathfrak{g}_i$ be a simple factor of $\mathfrak{g}$. The Lie algebra $\widetilde{L}(\mathfrak{g}_i, \sigma_i)\oplus \mathbb{F}d$ is an indecomposable, irreducible Kac-Moody algebra.
\item The map
\begin{displaymath}\varphi:\left(\widetilde{L}(\mathfrak{g}_a, \sigma_a)\oplus\bigoplus_{i=1}^n \widetilde{L}(\mathfrak{g}_i, \sigma_i)\right)\longrightarrow \widetilde{L}(\mathfrak{g}, \sigma)\,,\end{displaymath}
defined by $\varphi\left((f_a, r_a), (f_1, r_{c_1}), \dots, (f_n, r_{c_n})\right)=(f_a, f_1, \dots, f_n,  r_a+r_c=r_a+\sum r_{c_i})$ is a surjective Lie algebra homomorphism. In particular there is a short exact sequence
\begin{displaymath}1\longrightarrow \mathbb{F}^{n-1}\longrightarrow \left(\bigoplus_{i=1}^n \widetilde{L}(\mathfrak{g}_i, \sigma_i)\right) \longrightarrow \widetilde{L}(\mathfrak{g}, \sigma)\longrightarrow 1\,.\end{displaymath}
\end{enumerate}
\end{lemma}

\begin{proof}~
\begin{enumerate}
\item Recall the description by generators and relations (example~\ref{geometricaffinekacmoodyalgebra}); it is immediate, that $\widetilde{L}(\mathfrak{g}, \sigma)$ is an ideal.
\item To check the second assertion, it is sufficient to verify that the Lie bracket preserves the two components: let $f_i\in \widetilde{L}(\mathfrak{g}_i, \sigma_i)$, and $g+\mu d\in \widehat{L}(\mathfrak{g}, \sigma)$.  Then $[f_i, g+\mu d]= [f_i, g]-\mu f_i'$. $f_i'$ is in $\widetilde{L}(\mathfrak{g}_i, \sigma_i)$, the same is true for $[f_i, g]$, as it is true pointwise for elements in $\mathfrak{g}$.

\item 3.\ follows directly.

\item To prove the fourth claim, remark first that $\varphi|_{L(\mathfrak{g}, \sigma)}$ is an isomorphism. So we are left with checking the behavior of the extensions.
\begin{alignat*}{1}
\hspace{3pt}&\hspace{3pt}\varphi\left[\left((f_a, r_a), (f_1, r_{c_1}), \dots, (f_n, r_{c_n})\right); ((\bar{f}_a, \bar{r}_a), (\bar{f}_1, \bar{r}_{c_1}), \dots, \left(\bar{f}_n, \bar{r}_{c_n})\right)\right] = \\
= &\hspace{3pt}\varphi[f_a, f_1, \dots, f_n;\bar{f}_a, \bar{f}_1, \dots, \bar{f}_n]_0+\varphi\left(\omega_a(f_a, \bar{f}_a)+\sum_{i=1}^n\omega_i\left(f_i, \bar{f}_i'\right)c_i\right)=\\
= &\hspace{3pt}\left[\varphi(f_a, f_1, \dots, f_n); \varphi(\bar{f}_a, \bar{f}_1, \dots, (\bar{f}_n))\right]_0+\sum_{i=1}^n \omega\left(\varphi(f_a, f_1, \dots, f_n), \varphi(\bar{f}_a, \bar{f}_1, \dots , \bar{f}_n)\right)c=\\
= &\hspace{3pt}\left[\varphi((f_a, r_{a}), (f_1, r_{c_1}), \dots, (f_n, r_{c_n}));\varphi((\bar{f}_a, \bar{r}_{a}), (\bar{f}_1, \bar{r}_{c_1}), \dots, (\bar{f}_n, \bar{r}_{c_n}))\right]
\end{alignat*}
The existence of the short exact sequence is now immediate.\qedhere
\end{enumerate}

\end{proof}

\noindent Now we want to prove a similar splitting theorem for automorphisms of complex geometric affine Kac-Moody algebras.

\noindent The following lemma allows to restrict this problem to the case of loop algebras:

\begin{samepage}
\begin{lemma}~ 
\label{lemma_reduction_of_isomorphism_to_loop_algebra}
\begin{enumerate}
\item Any automorphism $\widehat{\varphi}:\widehat{L}(\mathfrak{g}_{\mathbb {R}},\sigma) \longrightarrow \widehat{L}(\mathfrak{g}_{\mathbb {R}}, \sigma)$ induces an automorphism of the derived algebras $\widetilde{\varphi}:\widetilde{L}(\mathfrak{g}_{\mathbb {R}}, \sigma)\longrightarrow \widetilde{L}(\mathfrak{g}_{\mathbb {R}}, \sigma)$.
\item Any automorphism $\widehat{\varphi}:\widehat{L}(\mathfrak{g}_{\mathbb {R}},\sigma) \longrightarrow \widehat{L}(\mathfrak{g}_{\mathbb {R}}, \sigma)$ induces an automorphism of the loop algebras $\varphi:L(\mathfrak{g}_{\mathbb {R}}, \sigma)\longrightarrow L(\mathfrak{g}_{\mathbb {R}}, \sigma)$.
\end{enumerate}
\end{lemma} 
\end{samepage}

\begin{proof}
Restrict the automorphism $\widehat{\varphi}$ to the subalgebra - see~\cite{Heintze09}, theorem 3.4.\,.
\end{proof}

\noindent For loop algebras of real or complex type we have the following decomposition results:

\begin{lemma}
\label{decomposition of loop algebras}
Let $\mathfrak{g}=\mathfrak{g}_s \oplus \mathfrak{g}_a$ and $(\mathfrak{g_s}, \sigma):= \bigoplus_i (\mathfrak{g_i}, \sigma_i)$, where $(\mathfrak{g_i}, \sigma_i)$ is a simple Lie algebra of the compact type together with the restriction $\sigma_i$ of $\sigma$ to $\mathfrak{g}_i$.
Let $L(\mathfrak g, \sigma)$ be the associated loop algebra and $\varphi$ a finite order automorphism of $L(\mathfrak g, \sigma)$. 
Then
\begin{enumerate}
\item $\varphi\left(L(\mathfrak{g}_a, \sigma)\right)=L(\mathfrak{g}_a, \sigma)$ and $\varphi\left(L(\mathfrak{g}_s, \sigma)\right)=L(\mathfrak{g}_s, \sigma)$.
\item $L(\mathfrak{g}_s, \sigma)$ decomposes under the action of $\varphi$ into $\varphi$-invariant ideals of two types:
   \begin{enumerate}
			\item Loop algebras of simple Lie algebras $L(\mathfrak{g}_i,\sigma_i)$ together with an automorphism $\varphi_i$ (called ``type $I$-factors''),
			\item Loop algebras of products of simple Lie algebras $\mathfrak{g_i}=\oplus_{i=1}^m\mathfrak{g_i}'$ together with an automorphism $\varphi_i$ of order $n$, cyclically interchanging the $m$-factors (called ``type $II$-factors''). In this case $\frac{n}{m}=k\in \mathbb{Z}$, and $\sigma$ induces an automorphism of order $k$ on each simple factor.
   \end{enumerate}
\end{enumerate}
\end{lemma}

\begin{proof}
\begin{enumerate}
\item Let $\mathfrak{g}=\mathfrak{g}_a \oplus \mathfrak{g}_s$. As remarked, a function $f:\mathbb{R}\longrightarrow \mathfrak{g}$ has a unique decomposition $f:=(f_a, f_s)$, such that $f_s: \mathbb{R}\rightarrow \mathfrak{g}_s$ and $f_a: \mathbb{R}\rightarrow \mathfrak{g}_a$. As $f(t+2\pi)=\sigma f(t)$ is equivalent to $f_a(t+2\pi)=\sigma f_a(t)=f_a(t)$ and $f_s(t+2\pi)=\sigma_s f_s(t)$ and $\sigma|_{\mathfrak{g}_a}=\Id$ this induces the decomposition: $L(\mathfrak{g}, \sigma)=L (\mathfrak{g}_a, \Id) \oplus L(\mathfrak{g}_s, \sigma_s)$. 
Let now $$\varphi: L(\mathfrak{g}, \sigma)\rightarrow L(\mathfrak{g}, \sigma)$$ be an automorphism. There are three (not mutually exclusive) possibilities:
 \begin{itemize}
 \item {\bf First case:}  $\varphi(L(\mathfrak{g}_a, \sigma))\subset L(\mathfrak{g}_a, \sigma)$ and $\varphi(L(\mathfrak{g}_s, \sigma))\subset L(\mathfrak{g}_s, \sigma)$. 
 \item {\bf Second case:} $\varphi(L(\mathfrak{g}_a, \sigma))\not\subset L(\mathfrak{g}_a, \sigma)$. 
  \item {\bf Third case:} $\varphi(L(\mathfrak{g}_s, \sigma))\not\subset L(\mathfrak{g}_s, \sigma)$. 
\end{itemize}
\begin{enumerate}  
\item In the {\bf first case}, we are done. 
\item {\bf We show: The second case and the third case are equivalent}:
     \begin{itemize}
     \item  {\bf We show: The second case includes the third case.} Assume, we are in the second case. By assumption we have $\varphi(L(\mathfrak{g}_a, \sigma))\cap L(\mathfrak{g}_{s}, \sigma)\not=\emptyset$. 
     Let $L(\mathfrak{g}_a, \sigma)=\mathcal{G}_a^1\oplus \mathcal{G}_a^2$ such that $\varphi(\mathcal{G}_{a}^{1})\subset L(\mathfrak{g}_{a}, \sigma)$ and $\varphi(\mathcal{G}_a^2)\subset L(\mathfrak{g}_s, \sigma)$. As $L(\mathfrak{g}_a, \sigma)$ and $L(\mathfrak{g}_s, \sigma)$ are ideals in $L(\mathfrak{g}_a, \sigma)$ we get, that $\mathcal{G}_a^1$ and $\mathcal{G}_a^2$ are ideals in $L(\mathfrak{g}_{a}, \sigma)$. Decompose $L(\mathfrak{g}_{s}, \sigma)=:\mathcal{G}_s^1\oplus\mathcal{G}_s^2$ such that $\mathcal{G}_s^2=\varphi(L(\mathfrak{g}_s, \sigma))$ and $\mathcal{G}_s^1=\varphi(L(\mathfrak{g}_a, \sigma))$. Then $\varphi(\mathcal{G}_{a}^{2})=\mathcal{G}_{s}^{2}$. As $\varphi$ has finite order, i.e. $\varphi^n=\Id$, it follows that $\varphi^{n-1}\mathcal{G}_s^2=\mathcal{G}_{a}^{2}$ and hence $\varphi(L(\mathfrak{g}_s, \sigma))\not\subset L(\mathfrak{g}_s, \sigma)$. Thus the second case includes the third case.
      \item {\bf  We show: The second case includes the third case.} This follows similar to the opposite direction.
      \end{itemize}
 In consequence the conditions 
      $\varphi(L(\mathfrak{g}_a, \sigma))\not\subset L(\mathfrak{g}_a, \sigma)$ and 
 $ \varphi(L(\mathfrak{g}_s, \sigma))\not\subset L(\mathfrak{g}_s, \sigma)$ are equivalent.

      \item {\bf We prove, that the third case is impossible.} To this end remark that $\left[L{\mathfrak{g}_s,\sigma}, L{\mathfrak{g}_s,\sigma}\right]=L{\mathfrak{g}_s,\sigma}$ while $\left[L{\mathfrak{g}_a,\sigma}, L{\mathfrak{g}_a,\sigma}\right]=0$. As  $\varphi$ is a Lie algebra homomorphism, it follows, that it maps brackets onto brackets. Hence $\varphi(L(\mathfrak{g}_s, \sigma))\cap L(\mathfrak{g}_s, \sigma)=\emptyset$. 
 \end{enumerate}

\item Let now $\mathfrak{g}_s=\bigoplus_{i=1}^m\mathfrak{g}_i$ be a decomposition of $\mathfrak{g}_s$ such that  
\begin{enumerate}	
 \item $\mathfrak{g}_i$ is invariant under $\varphi = \varphi|_{\mathfrak{g}_i}$.
 \item There is no decomposition $\mathfrak{g}_i=\mathfrak{g}'_i \oplus \mathfrak{g}''_i$ such that  $\varphi|_{\mathfrak{g}'_i}$ and $\varphi|_{\mathfrak{g}''_i}$ are automorphisms and $\mathfrak{g}'_i$ and $\mathfrak{g}''_i$ are invariant under the bracket operation.
\end{enumerate}

Again $f_s$ splits into $m$ component functions $f_s=(f_1, \dots, f_m)$ and the compatibility condition $f_s(t+2\pi)=\sigma f_s(t)$ is equivalent to the $m$ compatibility conditions $f_i(t+2\pi)=\sigma f_i(t), i=1, \dots, m$.

\noindent There are now two cases:
   \begin{enumerate}
   \item Suppose first, $\mathfrak{g}_i$ is simple. Then $\varphi_i$ is an involution of $L(\mathfrak{g}_i, \sigma_i)$. The pair 
      \begin{displaymath}\left(L(\mathfrak{g}_i, \sigma_i), \varphi_i\right)
      \end{displaymath} is of type $I$. 
     The finite order automorphisms of simple affine geometric Kac-Moody algebras are completely classified (see~\cite{Heintze09}).
   \item Suppose now $\mathfrak{g}_i$ is not simple and $\varphi|_{L(\mathfrak{g}_i,\sigma_i)}$ is of order $n$. There is a decomposition
         $\mathfrak{g}_i:=\bigoplus_j \mathfrak{g}_i^j$ such that $\mathfrak{g}_i^j$ is a simple Lie algebra. As there is no subalgebra
          $L(\mathfrak{h},\sigma_i|_{\mathfrak{h}})\subset L(\mathfrak{g}_i,\sigma_i)$ which is both invariant under $\varphi_{\mathfrak{g}_i}$ and an
           ideal in $L(\mathfrak{g}, \sigma)$, we find that all $(\mathfrak{g}_i^j, \sigma_i^j)$ are of the same type and the algebras $L(\mathfrak{g}_i^j, \sigma_i^j)$ are permuted by $\varphi_{\mathfrak{g}_i}$. Thus the number of those factors, denoted $m$, is a divisor of $n=km$.
           We get $$L(\mathfrak{g}_i, \sigma_i):=\bigoplus_{j=1}^m L(\mathfrak{g}_j,\sigma_j)\,.$$
$\varphi$ induces an automorphism $\bar{\varphi}$ of order $k$ on each simple factor. $\bar{\varphi}$ is again a finite order automorphism of a simple geometric affine Kac-Moody algebra. \qedhere
    \end{enumerate}
\end{enumerate}
\end{proof}

Using the classification result of E.\ Heintze and C.\ Gro\ss\ in~\cite{Heintze09} we know that every automorphism of the loop algebra $L(\mathfrak{g},\sigma)$ is of standard form
\begin{displaymath}\varphi\left( f(t)\right)=\varphi_t f(\lambda(t))\,.\end{displaymath}
Here $\varphi(t)$ denotes a curve of automorphisms of $\mathfrak{g}$ and $\lambda:\mathbb{R}\longrightarrow \mathbb{R}$ is a smooth function. Not all such automorphisms are extendible to the affine Kac-Moody algebra associated to $L(\mathfrak{g}, \sigma)$. We quote theorem 3.4.\ of~\cite{Heintze09}:

\begin{theorem}[Heintze-Gro\ss, 09]
\label{theorem:Heintze_Gross}
Let $\widehat{\varphi}:\widehat{L}(\mathfrak{g}, \sigma)\longrightarrow \widehat{L}(\widetilde{\mathfrak{g}}, \widetilde{\sigma})$  be a linear or conjugate linear map. Then $\widehat{\varphi}$ is an isomorphism of Lie algebras iff there exists $\gamma\in \mathbb{F}$ and a linear (resp.\ conjugate linear) isomorphism $\varphi:L(\mathfrak{g}, \sigma)\longrightarrow L(\widetilde{\mathfrak{g}}, \widetilde{\sigma})$ with $\lambda'=\epsilon_{\varphi}$ constant such that
\begin{alignat*}{2}
\widehat{\varphi}c&= \epsilon_{\varphi}c\\
\widehat{\varphi}d&=\epsilon_{\varphi}d-\epsilon_{\varphi} f_{\varphi}+\gamma c\\
\widehat{\varphi}f&=\varphi(u)+\mu(f) c\,.
\end{alignat*}
\end{theorem} 

Following~\cite{Heintze09} we call $\widehat{\varphi}$ of ``first type'' if $\epsilon_{\varphi}=1$ and of ``second type'' if $\epsilon_{\varphi}=-1$.
Similarly for geometric affine Kac-Moody algebras only those two types exist: We have the following result:

\begin{theorem}\
\label{theorem:admissible_locally_admissible}
\begin{enumerate}
\item Every involution $\varphi_i$ is of the first type or the identity. Then $\varphi$ is called of first kind.
\item Every involution $\varphi_i$ is of the second type. Then $\varphi$ is called of second kind.
\end{enumerate}
\end{theorem}

\begin{remark}
It is important to note that on every simple factor of a loop algebra we can choose any automorphism we want, especially it is possible to use any locally admissible automorphism --- i.e.\ the identity, automorphisms of first type and automorphism of second type --- simultaneously. This result shows that this does no longer hold for indecomposable (perhaps reducible) Kac-Moody algebras.
\end{remark}

Let $\mathfrak{g}$ be a reductive Lie algebra, $\sigma$ a diagram automorphism and $L(\mathfrak{g}, \sigma)$ the associated (twisted) loop algebra.
 
Call an involution $\varphi$ of $L(\mathfrak {g}, \sigma)$ \emph{locally admissible} iff its restriction to any irreducible factor can be extended to the associated affine Kac-Moody algebra; call it \emph{admissible} iff it can be extended to $\widehat{L}(\mathfrak {g}, \sigma)$. Remark that admissible involutions are locally admissible, but the converse fails in general.

Study now extensions to $\widehat{L}(\mathfrak{g}, \sigma)$. Let $\widehat{L}(\mathfrak{g}_i, \sigma_i)$ be any subalgebra. Assume the involution $\varphi_i$ of $L(\mathfrak{g}_i, \sigma)$ to be admissible; then it induces an involution $\widehat{\varphi}_i$ of $\widehat{L}(\mathfrak{g}_i, \sigma_i)$; $\widehat{\varphi}_i$ is unique as a consequence of the above (theorem~\ref{theorem:Heintze_Gross}) quoted result of~\cite{Heintze09}. Hence a locally admissible involution is admissible if the numbers $\epsilon_{\phi_i}$ coincide for each factor. This establishes the following lemma.

\begin{lemma} 
\label{lemma:loc_admissible_global_admissible}
A locally admissible involution of $\varphi:L(\mathfrak{g}, \sigma)\longrightarrow L(\mathfrak{g}, \sigma)$ is admissible, iff every restriction $\varphi_i:L(\mathfrak{g}_i, \sigma_i)\longrightarrow L(\mathfrak{g}_i, \sigma_i)$ has the same extension to $\widehat{L}(\mathfrak{g}_i, \sigma_i)$.
\end{lemma}

\noindent Theorem~\ref{theorem:admissible_locally_admissible} is now a direct consequence of lemma~\ref{lemma:loc_admissible_global_admissible}.

\section{Real forms of geometric affine Kac-Moody algebras}
\label{sect:real_forms}

In this section we will study real forms of geometric affine Kac-Moody algebras. The results in this section hold for twisted and non-twisted affine Kac-Moody algebras. We write twisted affine Kac-Moody algebras and use the convention that $\sigma$ may denote the identity.

\begin{example}
Let $A$ be an affine Cartan matrix, $\mathfrak{g}_{\mathbb{C}}(A)$ be its Lie algebra realization over $\mathbb{C}$. The realization $\mathfrak{g}_{\mathbb{R}}(A)$ over $\mathbb{R}$ is a real form of $\mathfrak{g}_{\mathbb{C}}(A)$. It is the fixed point algebra under the automorphism of complex conjugation. Similarly any completion $\widehat{L}(\mathfrak{g}_{\mathbb{R}}, \sigma)$ is a real form of the corresponding completion $\widehat{L}(\mathfrak{g}_{\mathbb{C}}, \sigma)$. 
\end{example}

\noindent The following lemma is straightforward:

\begin{lemma}
\label{lemma:real_form_product}
Let $\widehat{L}(\mathfrak{g}_{\mathbb{C}}, \sigma)$ be a geometric affine Kac-Moody algebra and let $\mathfrak{g}_{\mathbb{C}}=\oplus \mathfrak{g}_{\mathbb{C},i}$. Then any real form $\widehat{L}_{\textrm{real}}(\mathfrak{g}_{\mathbb{C}}, \sigma)$ satisfies

 $$\widehat{L}_{\textrm{real}}(\mathfrak{g}_{\mathbb{C}}, \sigma)\cap \widehat{L}(\mathfrak{g}_{\mathbb{C},i}, \sigma_i)=\widehat{L}_{\textrm{real}}(\mathfrak{g}_{\mathbb{C},i}, \sigma_i)\,.$$ 
 
\noindent The same is true for the loop algebra.
\end{lemma}

Hence a real form of a reducible complex affine geometric Kac-Moody algebra $\widehat{L}(\mathfrak{g}_{\mathbb{C}}, \sigma)$ induces real forms of each of its irreducible factors $\widehat{L}(\mathfrak{g}_{\mathbb{C},i}, \sigma_i)$.

We have described in section~\ref{sect:geometric_kac_moody} that involutions of a geometric affine Kac-Moody algebra restrict to involutions of irreducible factors of the loop algebra. Hence, the invariant subalgebras are direct products of invariant subalgebras in those factors together with the appropriate $2$-dimensional extension.

\begin{definition}
A \emph{compact real form} of a complex affine Kac-Moody algebra $\widehat{L}(\mathfrak{g}_{\mathbb{C}}, \sigma)$ is defined to be a subalgebra of $\widehat{L}(\mathfrak{g}_\mathbb{C},\sigma)$ that is conjugate to the algebra
$\widehat{L}(\mathfrak{g}_{\mathbb {R}}, \sigma)$
where $\mathfrak{g}_{\mathbb R}$ is a compact real form of
$\mathfrak{g}_{\mathbb C}$. 
\end{definition}

To find non-compact real forms, we need the following result of E.\ Heintze and C.\ Gro\ss\ (Corollary 7.7.\ of~\cite{Heintze09}):

\begin{theorem}
\label{theoremofHeintzegross}
Let $\mathcal{G}$ be an irreducible complex geometric affine Kac-Moody algebra, $\mathcal{U}$ a real form of compact type. The conjugacy classes of real forms of non-compact type of $\mathcal{G}$ are in bijection with the conjugacy classes of involutions on $\mathcal{U}$. The correspondence is given by $\mathcal{U}=\mathcal{K}\oplus \mathcal{P}\mapsto \mathcal{K}\oplus i\mathcal{P}=:\mathcal{U}^{*}$ where $\mathcal{K}$ and $\mathcal{P}$ are the $\pm 1$-eigenspaces of the involution.  
\end{theorem}

Thus to find non-compact real forms, we have to study automorphisms of order $2$ of a geometric affine Kac-Moody algebra of the compact type. From now on we restrict to involutions $\widehat{\varphi}$ of second type, that is those such that $\epsilon_{\varphi}=-1$.

Now we want to extend theorem \ref{theoremofHeintzegross} to non-irreducible geometric affine Kac-Moody algebras.
\begin{enumerate}
\item Suppose first that the involution $\widehat{\varphi}$ on $\widehat{L}(\mathfrak{g}, \sigma)$ is chosen in a way that every irreducible factor is of type $I$. In this case $\widehat{\varphi}$ restricts to an involution $\widehat{\varphi}_i$ on every irreducible factor $\widehat{L}(\mathfrak{g}_i, \sigma_i)$. Let $L(\mathfrak{g}_i, \sigma_i)=\mathcal{K}_i \oplus \mathcal{P}_i$ be the decomposition of $L(\mathfrak{g}_i, \sigma_i)$ into the eigenspaces of $\varphi_i$. The dualization procedure follows the standard pattern, defining $\mathcal{K}=\oplus_i \mathcal{K}_i$ and $\mathcal{P}=\oplus_i \mathcal{P}_i\oplus \mathbb {R} c\oplus \mathbb{R}d$ and exchanging $\mathcal{U}=\mathcal{K}\oplus \mathcal{P}$ by $\mathcal{U}^*=\mathcal{K}\oplus i\mathcal{P}$.

\item If the decomposition of $\widehat{\varphi}$ contains simple factors of type $II$, we perform the same decomposition procedure: 
let $\widehat{L}(\mathfrak{g}_j \oplus \mathfrak{g}_j', \sigma_j \oplus \sigma_j')$ be irreducible with respect to $\widehat{\varphi}$, then we have the decomposition 
$L(\mathfrak{g}_j\oplus \mathfrak{g}_j', \sigma_j \oplus \sigma_j')=\mathcal{K}_j\oplus \mathcal{P}_j$. Dualization follows the standard pattern.
\end{enumerate}

\noindent We have to investigate iff all real forms can be described in this way:

\begin{enumerate}
\item If $\mathfrak{g}$ is simple it is a result of E.\ Heintze and C.\ Gro\ss\ that every real form of non-compact type can be constructed in this way (see~\cite{Heintze09}, section 7).
\item If $\mathfrak{g}$ is not simple we use  that the restriction to the loop algebra $L(\mathfrak{g}, \sigma)$ of any real form consists of the direct product of real forms in the irreducible components of $L(\mathfrak{g}_{\mathbb{C}}, \sigma)$. Hence, according to the result of E.\ Heintze and C.\ Gro\ss, those are of the described type, and thus the non-compact real form we started with.
\end{enumerate}

\noindent Hence we have established the following results

\begin{theorem}
Let $\widehat{L}(\mathfrak{g}, \sigma)$ be a geometric affine Kac-Moody algebra of the compact type. Let $\widehat{\varphi}$ be an involution of order $2$ of the second kind of $\widehat{L}(\mathfrak{g}, \sigma)$. Let furthermore  $\widehat{L}(\mathfrak{g}, \sigma)=\mathcal{K} \oplus \mathcal{P}$ be the decomposition into its $\pm 1$-eigenspaces.
Then $\mathcal{G}_D:=\mathcal{K}\oplus i \mathcal{P}$ is the dual real form of the non-compact type. 
\end{theorem}

\begin{theorem}
\label{eithercompactornoncompact}
Every real form of a complex geometric affine Kac-Moody algebra is either of compact type or of non-compact type. A mixed type is not possible.
\end{theorem}

\section{Compact real forms of loop algebras}
\label{sect:compact_loop}

A finite dimensional Lie algebra $\mathfrak{g}$ is of compact type iff its  Cartan-Killing form is negative definite (see~\cite{Helgason01} II.6.6). A similar criterion applies to loop algebras. For affine Kac-Moody algebras an analogous bilinear form may be naturally defined as follows:

\begin{definition}
The \emph{Cartan-Killing form} of a loop algebra $L(\mathfrak{g},\sigma)$ is defined by
\begin{displaymath}B_{L(\mathfrak{g},\sigma)}\left(f, g\right)=\int_{0}^{2\pi} B_{\mathfrak{g}}\left(f(t), g(t) \right)dt\,.\end{displaymath}
where $B_{\mathfrak{g}}\left(f(t), g(t) \right)$ denotes the Cartan-Killing form of the finite dimensional Lie algebra $\mathfrak{g}$
\end{definition}

\begin{definition}
A loop algebra of compact type is a real form of $L(\mathfrak{g}_{\mathbb{C}},\sigma)$ such that its Cartan-Killing form is negative definite. 
\end{definition}

\begin{proposition}
Let $\mathfrak{g}$ be Abelian. The Cartan-Killing form of $L(\mathfrak{g})$ is trivial.
\end{proposition}

\begin{proof}
If $\mathfrak{g}$ is Abelian its Cartan-Killing form $B_{\mathfrak{g}}$ vanishes. Hence the integral over $B_{\mathfrak{g}}$ vanishes.
\end{proof}

\begin{lemma}
\label{compact_type_negative_killing}
Let $\mathfrak{g}_{\mathbb{R}}$ be a compact semisimple Lie algebra.
Then the loop algebra $L(\mathfrak{g}_{\mathbb{R}}, \sigma)$ is of compact type.
\end{lemma}

\begin{proof}
By definition, the Cartan-Killing form $B_{\mathfrak{g}}$ on the finite dimensional Lie algebra $\mathfrak{g}_{\mathbb{R}}$ is negative definite. 
Now by the finite dimensional result, $c(t):=B_{\mathfrak{g}}\left(f(t), f(t)\right)<0$.  Then  
\begin{displaymath}B_{L(\mathfrak{g},\sigma)}\left(f, f\right)=\int_{0}^{2\pi} B_{\mathfrak{g}}\left(f(t), f(t) \right)dt=\int_0^{2\pi} c(t) dt<0\,.\end{displaymath}
is strictly negative definite.
\end{proof}

\begin{definition}
Let $\mathfrak{g}_{\mathbb{R}}$ be a compact semisimple Lie algebra.
The loop algebra $L(\mathfrak{g}_{\mathbb{R}}, \sigma)$ is called the standard compact real form.
\end{definition}

\noindent In general we have the following result (proven in~\cite{Heintze09}, theorem 7.4):
 
\begin{lemma}
\label{unicity_of_compact_real_forms}
Let $\mathfrak{g}$ be simple. $\widehat{L}(\mathfrak{g}, \sigma)$ and ${L}(\mathfrak{g}, \sigma)$ have a compact real form, which is unique up to conjugation. 
\end{lemma}

\noindent As the $\Ad$-invariant scalar product is invariant under conjugation, lemma~\ref{compact_type_negative_killing} has the following corollary:

\begin{corollary}
Let $\mathfrak{g}$ be a simple, complex Lie algebra. Any compact real form of the loop algebra ${L}(\mathfrak{g}, \sigma)$ has a
negative definite Cartan-Killing form.
\end{corollary}

\noindent We extend these results to geometric affine Kac-Moody algebras:

\begin{theorem}
 Let $\mathfrak{g}$ be a reductive Lie algebra. The following assertions are equivalent:
 \begin{itemize}
\item $\mathfrak{g}$ is semisimple;
\item the loop algebra $L(\mathfrak{g}_{\mathbb{C}},\sigma)$ has a compact real forms which is unique up to conjugation;
\item the geometric affine Kac-Moody algebra $\widehat{L}(\mathfrak{g}_{\mathbb{C}},\sigma)$ has a compact real forms which is unique up to conjugation.
\end{itemize}
\end{theorem}

\begin{proof}
\begin{itemize}
\item Let $\mathfrak{g}_{\mathfrak{C}}:=\bigoplus_{i=0}^n \mathfrak{g}_i$ such that $\mathfrak{g}_0$ is Abelian and $\mathfrak{g}_i$ is simple for $i\geq 0$. We study first the loop algebra case: By lemma~\ref{lemma:splitting} we have the decomposition
$$L(\mathfrak{g}_{\mathfrak{C}}, \sigma)=\bigoplus_{i=0}^n L(\mathfrak{g}_i, \sigma_i)\,.$$
By lemma~\ref{lemma:real_form_product} any real form of $L_{\textrm{real}}(\mathfrak{g}, \sigma)\subset L(\mathfrak{g}_{\mathfrak{C}}, \sigma)$ decomposes into a direct sum 
$$L_{\textrm{real}}(\mathfrak{g}, \sigma)=\bigoplus_{i=0}^n L_{\textrm{real}}(\mathfrak{g}_i, \sigma_i)\,.$$
Lemma~\ref{unicity_of_compact_real_forms} assures for each factor $L(\mathfrak{g}_i, \sigma_i)$ for $i\geq 1$ and $\mathfrak{g}_i$ simple the existence of some (up to conjugation unique) compact real form $L_{c}(\mathfrak{g}_i, \sigma_i)$. In contrast, as $\mathfrak{g}_0$ is Abelian, then the Cartan-Killing form vanishes identically; hence the algebra $L(\mathfrak{g}_{0})$ does not have a compact real form. Thus we have proven that $L(\mathfrak{g}_{\mathbb{C}},\sigma)$ has a compact real form, which is unique up to conjugation iff $\mathfrak{g}_0=0$. 
\item  To extend the result to geometric affine Kac-Moody algebras, we have to check that the compact real form allows for an extensions by $c$ and $d$.  This follows from the explicit construction: Conjugating the compact loop algebra into the standard compact real form $\widehat{L}(\mathfrak{g}_c, \sigma)$ where $\mathfrak{g}_c$ denotes the compact real form. Then the extension can be constructed for each factor separately: We get in each factor imaginary coefficients, and hence extensions $i\mathbb{R}c\oplus i\mathbb{R}d$.   
\end{itemize} 
\end{proof}

\begin{proposition}
Let $\mathfrak{g}$ be semisimple and $\widehat{L}(\mathfrak{g}, \sigma)_D$ be a real form of the non-compact type. Let $\widehat{L}(\mathfrak{g}, \sigma)_D=\mathcal{K}\oplus \mathcal{P}$ be a Cartan decomposition. The Cartan-Killing form is negative definite on $\mathcal{K}$ and positive definite on $\mathcal{P}$.
\end{proposition}

\begin{proof}
Suppose first $\sigma$ is the identity. Let $\varphi:L(\mathfrak{g}, \sigma)\longrightarrow L(\mathfrak{g}, \sigma)$ be an automorphism. Then without loss of generality $\varphi(f)=\varphi_0(f(-t))$ (see~\cite{Heintze09}, section 2). Let $\mathfrak{g}=\mathfrak{k}\oplus \mathfrak{p}$ be the decomposition of $\mathfrak{g}$ into the $\pm 1$-eigenspaces of $\varphi_0$. Then
$f\in \textrm{Fix}(\varphi)$ iff its Taylor expansion satisfies
$$\sum_n a_n e^{int}= \sum_n \varphi_0(a_{-n})e^{int}\,.$$
Let $a_n=k_n\oplus p_n$ be the decomposition of $a_n$ into the $\pm 1$ eigenspaces with respect to $\varphi_0$. 
Hence $$f(t)=\sum_n k_n \cos(nt)+\sum_n p_n \sin(nt)\,.$$ 
Then, using bilinearity and orthonormality of $\{\cos(nt), \sin(nt)\}$,  we can calculate $B_{\mathfrak{g}}$:
$$B_\mathfrak{g}=\int_0^{2\pi}\sum_n \cos^{2}(nt)B(k_n, k_n)-\int_0^{2\pi}\sum_n \sin^{2}(nt)B(p_n, p_n)\,.$$
Hence $B_{\mathfrak{g}}$ is negative definite on $\textrm{Fix}(\varphi)$. Analogously one calculates that $B_{\mathfrak{g}}$ is positive definite on the $-1$-eigenspace of $\varphi$. 
If $\sigma\not=Id$ then a twisted loop algebra can always be understood as a subalgebra of an untwisted one, fixed by some (diagram) automorphism (see our discussion in section 2 and for additional details~\cite{Carter05}, pp.~451). This construction holds in the category of loop algebras as in the category of Kac-Moody algebras. It is then a straightforward verification, that the $\mathcal{K}$-part (resp. $\mathcal{P}$-part) of the twisted algebra is mapped into the $\mathcal{K}$-part (resp. $\mathcal{P}$-part) of the non-twisted algebra. Hence the results about the Killing form extend. 
\end{proof}

\section{Orthogonal symmetric affine Kac-Moody algebras}
\label{sect:osaka}

In this section we introduce and study \emph{o}rthogonal \emph{s}ymmetric \emph{a}ffine \emph{K}ac-Moody \emph{a}lgebras (OSAKAs) and describe their classification:

\begin{definition}
An \emph{orthogonal symmetric affine Kac-Moody algebra (OSAKA)} is a pair $\left(\widehat{L}(\mathfrak{g}, \sigma), \widehat{L}(\rho)\right)$ such that
\begin{enumerate}
	 \item $\widehat{L}(\mathfrak{g}, \sigma)$ is a real form of an affine geometric Kac-Moody algebra,
	 \item $\widehat{L}(\rho)$ is an involutive automorphism of $\widehat{L}(\mathfrak{g}, \sigma)$,
	 \item If $\mathfrak{g}=\mathfrak{g}_s\oplus \mathfrak{g}_a$ then the intersection $\textrm{Fix}(\widehat{L}(\rho))\cap{\mathfrak{g}_s}$ is a compact Kac-Moody algebra or a compact loop algebra and the intersection $\textrm{Fix}(\widehat{L}(\rho))\cap{\mathfrak{g}_a}=0$.
\end{enumerate}
\end{definition}

\begin{definition}
An OSAKA is called effective if $\textrm{Fix}(\widehat{L}(\rho))\cap\mathfrak{z}=0$ where $\mathfrak{z}$ denotes the center of $\widehat{L}(\mathfrak{g}, \sigma)$.
\end{definition}

\begin{lemma}
Let $\left(\widehat{L}(\mathfrak{g}, \sigma), \widehat{L}(\rho)\right)$ be an effective OSAKA. Then $\widehat{L}(\rho)$ is of second type and $\textrm{Fix}(\widehat{L}(\rho))$ is a compact loop algebra. 
\end{lemma}

\begin{proof}
By definition $c$ is central - hence $c\in \mathfrak{z}$. Hence for the OSAKA $\left(\widehat{L}(\mathfrak{g}, \sigma), \widehat{L}(\rho)\right)$ to be effective, we need $c\not\in \textrm{Fix}\left(\widehat{L}(\rho)\right)$. From the classification of involutions of affine Kac-Moody algebras we know that $\widehat{L}(\rho)(c)\in \{\pm c\}$. Hence $\widehat{L}(\rho)(c)=-c$ and the involution $\widehat{L}(\rho)$ is of second type.
By the classification of involutions (see~\cite{Heintze09}, section 6)  $\widehat{L}(\rho)(c)=-c$ implies that $\widehat{L}(\rho)$ is conjugate to an involution satisfying $\widehat{L}(\rho)(d)=-d$. Hence $\textrm{Fix}(\widehat{L}(\rho))\subset L(\mathfrak{g}, \sigma)$.
\end{proof}

OSAKAs appear in various kinds: the main distinction is as in the finite dimensional case the one between OSAKAs of the compact type and OSAKAs of the non-compact type. 
In between, we have the special case of OSAKAs of Euclidean type.

\begin{definition}
\label{types of OSAKAs}
Let $(\widehat{L}(\mathfrak{g}, \sigma), \widehat{L}(\rho))$ be an OSAKA. Let $\widehat{L}(\mathfrak{g}, \sigma)=\mathcal{K}\oplus \mathcal{P}$ be the decomposition of $\widehat{L}(\mathfrak{g}, \sigma)$ into the eigenspaces of $\widehat{L}(\mathfrak{g}, \sigma)$ of eigenvalue $+1$ resp.\ $-1$.
\begin{enumerate}
	\item If $\widehat{L}(\mathfrak{g}, \sigma)$ is a compact real affine Kac-Moody algebra, it is said to be of the compact type.
	\item If $\widehat{L}(\mathfrak{g}, \sigma)$ is a non-compact real affine Kac-Moody algebra, $\widehat{L}(\mathfrak{g}, \sigma)=\mathcal{U}\oplus \mathcal{P}$ is a Cartan decomposition of $\widehat{L}(\mathfrak{g}, \sigma)$.
	\item If $L(\mathfrak{g},\sigma)$ is Abelian, it is said to be of Euclidean type.
\end{enumerate}
\end{definition}

\begin{definition}
An OSAKA $(\widehat{L}(\mathfrak{g}, \sigma), \widehat{L}(\rho))$ is called \emph{semisimple} if $\widehat{L}(\mathfrak{g}, \sigma)$ is a semisimple geometric affine Kac-Moody algebra.
\end{definition}

OSAKAs of the compact type and of the non compact type are semisimple, OSAKAs of the Euclidean type are not semisimple.

\begin{definition}
An OSAKA $(\widehat{L}(\mathfrak{g}, \sigma), \widehat{L}(\rho))$ is called \emph{irreducible} iff it has no non-trivial ideal which is isomorphic to a Kac-Moody subalgebra and invariant under $\widehat{L}(\rho)$.
\end{definition}

\section{Irreducible OSAKAs}

In this section we describe the classification of irreducible OSAKAs.

\subsection*{Irreducible OSAKAs of the compact type}

We describe first the irreducible OSAKAs of compact type. They consist of two classes: 

\begin{itemize}
\item {\bf OSAKAs of type \boldmath$I$:}
 The first class consists of compact real forms $\widehat{L}(\mathfrak{g},{\sigma})$, where $\mathfrak{g}$ is a simple compact real Lie algebra together with an involution of the second kind $\widehat{\mathfrak{\rho}}$. As $\widehat{\mathfrak{\rho}}(c)=-c$ and $\widehat{\mathfrak{\rho}}(d)=-d$ up to conjugation, the fixed point algebra $\textrm{Fix}(\widehat{\mathfrak{\rho}})\subset L(\mathfrak{g},{\sigma})$. Hence the pair $\left(\widehat{L}(\mathfrak{g},{\sigma}), \widehat{L}(\rho)\right)$ is an OSAKA. A complete classification (up to conjugation) can be given as follows: As each Kac-Moody algebra has some (up to conjugation unique) compact real form, a classification consists in running through all pairs consisting of irreducible affine geometric Kac-Moody algebras $\widehat{L}(\mathfrak{g},{\sigma})$ and all conjugate linear involutions of the second kind of $\widehat{L}(\mathfrak{g},{\sigma})$.
 Irreducible affine geometric Kac-Moody algebras are in bijection with affine Kac-Moody algebras. A complete list consists thus of the algebras of the (non twisted) types $\mathfrak{a}_{n}^{(1)}, n\geq 1$, $\mathfrak{b}_{n}^{(1)}, n\geq 2$, $\mathfrak{c}_{n}^{(1)}, n\geq 3$, $\mathfrak{d}_n^{(1)}, n\geq 3$, $\mathfrak{e}_6^{(1)}$, $\mathfrak{e}_{7}^{(1)}$, $\mathfrak{e}_{8}^{(1)}$, $\mathfrak{f}_{4}^{(1)}$, $\mathfrak{g}_{2}^{(1)}$ and of the algebras of the (twisted) types $\widetilde{\mathfrak{a}}_{1}'$, $\widetilde{\mathfrak{c}}_{l}'$, $\widetilde{\mathfrak{b}}_l^t$, $\widetilde{\mathfrak{c}}_{l}^t$, $\widetilde{\mathfrak{f}}_{4}^t$, and $\widetilde{\mathfrak{g}}_{2}^t$. 
 
 A complete list of these involutions for the above algebras is given in~\cite{Heintze09}, section 6. For the infinite series the number of involutions grows roughly proportional to $n^2$ (with the only exception of $\mathfrak{a}_{n}^{(2)}$).  $\mathfrak{e}_6^{(1)}$ has $9$ involutions of second kind, $\mathfrak{e}_{7}^{(1)}$ has $10$ involutions of second kind, $\mathfrak{e}_{8}^{(1)}$ has $6$ involutions of second kind, $\mathfrak{f}_{4}^{(1)}$ has $6$ involutions of second kind, and $\mathfrak{g}_{2}^{(1)}$ has $3$ involutions of second kind.

\item {\bf OSAKAs of type \boldmath$II$:}
Let $\mathfrak{g}_{\mathbb{R}}$ be a simple real Lie algebra of the compact type. The second class consists of pairs of an affine Kac-Moody algebra
$\widehat{L}(\mathfrak{g}_{\mathbb{R}}\times \mathfrak{g}_{\mathbb{R}})$ together with an involution $\widehat{\varphi}$ of the second kind, that switches the two factors. 
\begin{align*}
\varphi:\quad \widehat{L}(\mathfrak{g}_{\mathbb{R}}\times \mathfrak{g}_{\mathbb{R}}, \sigma\oplus\sigma)&\longrightarrow \widehat{L}(\mathfrak{g}_{\mathbb{R}}\times \mathfrak{g}_{\mathbb{R}},\sigma\oplus\sigma)\\
(f(t), g(t), r_c, r_d)&\mapsto (g(-t), f(-t), -r_c, -r_d)
\end{align*}
Hence the fixed point algebra consists of elements $(f(t), f(-t), 0,0)$, which is isomorphic to $L(\mathfrak{g}, \sigma)$.
 As each Kac-Moody algebra has a (unique up to conjugation) compact real form, a complete classification consists in running through all irreducible geometric affine Kac-Moody algebras. Irreducible affine geometric Kac-Moody algebras are in bijection with affine Kac-Moody algebras. A complete list consists thus of the algebras of the (non twisted) types $\mathfrak{a}_{n}^{(1)}, n\geq 1$, $\mathfrak{b}_{n}^{(1)}, n\geq 2$, $\mathfrak{c}_{n}^{(1)}, n\geq 3$, $\mathfrak{d}_n^{(1)}, n\geq 3$, $\mathfrak{e}_6^{(1)}$, $\mathfrak{e}_{7}^{(1)}$, $\mathfrak{e}_{8}^{(1)}$, $\mathfrak{f}_{4}^{(1)}$, $\mathfrak{g}_{2}^{(1)}$ and of the algebras of the (twisted) types $\widetilde{\mathfrak{a}}_{1}'$, $\widetilde{\mathfrak{c}}_{l}'$, $\widetilde{\mathfrak{b}}_l^t$, $\widetilde{\mathfrak{c}}_{l}^t$, $\widetilde{\mathfrak{f}}_{4}^t$, and $\widetilde{\mathfrak{g}}_{2}^t$. 
\end{itemize}

\begin{proposition}
\label{prop:osakas_compact}
The OSAKAs described above are precisely all irreducible OSAKAs of compact type.
\end{proposition}

\begin{proof}
Remark first that all OSAKAs listed above are irreducible. For OSAKAs of type $I$ this is due to the fact, that the affine geometric Kac-Moody algebras $\widehat{L}(\mathfrak{g}, \sigma)$ for $\mathfrak{g}$ simple are irreducible Kac-Moody algebras, for OSAKAs of type $II$ this is a consequence of the involution $\widehat{L}(\mathfrak{\rho})$ switching the two factors.

The crucial part of the proposition is thus, that this list includes all irreducible OSAKAs. Let $\left(\widehat{L}(\mathfrak{g}, \sigma), \widehat{L}(\rho)\right)$ be an OSAKA. 
\begin{itemize}
\item Assume that $\widehat{L}(\mathfrak{g}, \sigma)$ is irreducible. Then we are done as any automorphism has to be of second kind. 
\item So assume that $\widehat{L}(\mathfrak{g}, \sigma)$ is reducible. Let $\widehat{L}(\mathfrak{g}_i, \sigma)\subset \widehat{L}(\mathfrak{g}, \sigma)$ be an irreducible factor and let $\widehat{L}_{\widehat{L}(\rho)}(\mathfrak{g}_i, \sigma)$ be the orbit of $\widehat{L}(\mathfrak{g}_i, \sigma)$ in $\widehat{L}(\mathfrak{g}, \sigma)$ under $\widehat{L}(\rho)$. As $\widehat{L}(\mathfrak{g}, \sigma)$ is irreducible, we have that 

$$\widehat{L}_{\widehat{L}(\rho)}(\mathfrak{g}_i, \sigma)=\widehat{L}(\mathfrak{g}, \sigma)\,.$$

\noindent $\widehat{L}(\rho)$ being an involution, we know from lemma~\ref{lemma_reduction_of_isomorphism_to_loop_algebra}, that it induces an involution on the level of loop algebras ; hence using lemma~\ref{decomposition of loop algebras}, we get:

$$\widehat{L}_{\widehat{L}(\rho)}(\mathfrak{g}_i, \sigma)\subseteq \widehat{L}(\mathfrak{g}_i\oplus \mathfrak{g}_i, \sigma)\,.$$
Hence  $\left(\widehat{L}(\mathfrak{g}, \sigma), \widehat{L}(\rho)\right)$ is of type $II$.
\end{itemize}
\end{proof}


\subsection*{Irreducible OSAKAs of the non-compact type}
Besides the OSAKAs of the compact type there are the OSAKAs of the non-compact type.

\begin{itemize}
\item{\bf OSAKAs of type \boldmath$III$:} Let $\mathfrak{g}_{\mathbb{C}}$ be a complex semisimple Lie algebra, $\mathfrak{g}$ its compact real form and $\widehat{L}(\mathfrak{g}_{\mathbb{C}},\sigma)$ (resp. $\widehat{L}(\mathfrak{g},\sigma)$) the associated complex affine Kac-Moody algebra (resp. compact form).
This class consists of real forms of the non-compact type that are described as fixed point sets of involutions of second kind together with a special involution, called the \emph{Cartan involution}. This is the unique involution on $\widehat{L}(\mathfrak{g},\sigma)$, such that the decomposition into its $\pm 1$-eigenspaces $\mathcal{K}$ and $\mathcal{P}$ yields: $\mathcal{K}\oplus i \mathcal{P}$ is a real form of compact type of $\widehat{L}(\mathfrak{g}_{\mathbb{C}},\sigma)$. For a proof of the existence and uniqueness of the Cartan involution see~\cite{Heintze09}. 
A complete list consists thus of the algebras of the (non twisted) types $\mathfrak{a}_{n}^{(1)}, n\geq 1$, $\mathfrak{b}_{n}^{(1)}, n\geq 2$, $\mathfrak{c}_{n}^{(1)}, n\geq 3$, $\mathfrak{d}_n^{(1)}, n\geq 3$, $\mathfrak{e}_6^{(1)}$, $\mathfrak{e}_{7}^{(1)}$, $\mathfrak{e}_{8}^{(1)}$, $\mathfrak{f}_{4}^{(1)}$, $\mathfrak{g}_{2}^{(1)}$ and of the algebras of the (twisted) types $\widetilde{\mathfrak{a}}_{1}'$, $\widetilde{\mathfrak{c}}_{l}'$, $\widetilde{\mathfrak{b}}_l^t$, $\widetilde{\mathfrak{c}}_{l}^t$, $\widetilde{\mathfrak{f}}_{4}^t$, and $\widetilde{\mathfrak{g}}_{2}^t$.  A complete list of non-compact real form of the second kind of the above algebras coincides with a complete list of involutions of second kind for the above algebras - we refer to~\cite{Heintze09}, section 6. For the infinite series the number of involutions grows roughly proportional to $n^2$ (with the only exception of $\mathfrak{a}_{n}^{(2)}$).  $\mathfrak{e}_6^{(1)}$ has $9$ involutions of second kind, $\mathfrak{e}_{7}^{(1)}$ has $10$ involutions of second kind, $\mathfrak{e}_{8}^{(1)}$ has $6$ involutions of second kind, $\mathfrak{f}_{4}^{(1)}$ has $6$ involutions of second kind, and $\mathfrak{g}_{2}^{(1)}$ has $3$ involutions of second kind (see~\cite{Heintze09}).

\item{\bf OSAKAs of type \boldmath$IV$:} Let $\mathfrak{g}_{\mathbb{C}}$ be a complex semisimple Lie algebra. The fourth class consists of negative-conjugate real forms of $\widehat{L}(\mathfrak{g}_{\mathbb{C}}, \sigma)$. The involution is given by the complex conjugation $\widehat{L}(\rho_0)$ with respect to a compact real form of $L(\mathfrak{g}_{\mathbb{C}},\sigma)\cong L(\mathfrak{g}, \sigma)\oplus i L(\mathfrak{g}, \sigma)$. Hence they have the form: $L(\mathfrak{g}, \sigma)\oplus iL(\mathfrak{g},\sigma)\oplus\mathbb{R}c\oplus \mathbb{R}d$.
Let us remark that we can also take the complex Kac-Moody algebra $\widehat{L}(\mathfrak{g}_{\mathbb{C}}, \sigma)$ and define the involution $\widehat{\rho}_0$ as conjugation with respect to the real form of compact type $\widehat{L}(\mathfrak{g}_{\mathbb{R}}, \sigma)$. Nevertheless the pair $(\widehat{L}(\mathfrak{g}_{\mathbb{C}}, \sigma), \widehat{\rho}_0)$ is not an OSAKA in the sense of our definition as the fixed point algebra is not a loop algebra of the compact type and as the subjacent Kac-Moody algebra is not a real Kac-Moody algebra. Both of those algebras describe the same symmetric space as their $\mathcal{P}$-components are the same.
Irreducible complex affine geometric Kac-Moody algebras are in bijection with affine Kac-Moody algebras. A complete list consists thus of the algebras of the (non twisted) types $\mathfrak{a}_{n}^{(1)}, n\geq 1$, $\mathfrak{b}_{n}^{(1)}, n\geq 2$, $\mathfrak{c}_{n}^{(1)}, n\geq 3$, $\mathfrak{d}_n^{(1)}, n\geq 3$, $\mathfrak{e}_6^{(1)}$, $\mathfrak{e}_{7}^{(1)}$, $\mathfrak{e}_{8}^{(1)}$, $\mathfrak{f}_{4}^{(1)}$, $\mathfrak{g}_{2}^{(1)}$ and of the algebras of the (twisted) types $\widetilde{\mathfrak{a}}_{1}'$, $\widetilde{\mathfrak{c}}_{l}'$, $\widetilde{\mathfrak{b}}_l^t$, $\widetilde{\mathfrak{c}}_{l}^t$, $\widetilde{\mathfrak{f}}_{4}^t$, and $\widetilde{\mathfrak{g}}_{2}^t$\,.. 
\end{itemize}

\begin{proposition}
The OSAKAs described above are precisely all irreducible OSAKAs of non-compact type.
\end{proposition}

\begin{proof}
The direct proof parallels the one of proposition~\ref{prop:osakas_compact}. 
Alternatively the result can be deduced by duality (see below).
\end{proof}

\subsection*{Duality between the compact type and the non-compact type}
\label{sect:duality}

Let us investigate the duality a little more closely. We follow the presentation for finite dimensional Lie algebras given in~\cite{Helgason01}.
Let $(\mathcal{G},\rho)$ be an OSAKA and let $\mathcal{G}=\mathcal{K}\oplus\mathcal{P}$ be the decomposition into the $\pm 1$-eigenspaces of $\rho$. Then $\mathcal{G}^*=\mathcal{K}\oplus\mathcal{P}$ is a real Lie form of $\mathcal{G}_{\mathbb{C}}$. For any element $g\in \mathcal{G}$ let $g=k+p$ be the decomposition into the $\mathcal{K}$- and $\mathcal{P}$-component. Define $\rho^*:k+ip\mapsto k-ip$.

\begin{proposition}
Let $(\mathcal{G},\rho)$ be an OSAKA. 
\begin{enumerate}
\item Then the pair $(\mathcal{G}^*, \rho^*)$ is an OSAKA, called the dual OSAKA.
\item If $(\mathcal{G},\rho)$ is of compact type then $(\mathcal{G}^*, \rho^*)$ is of non-compact type and vice versa.
\end{enumerate}
\end{proposition}

\begin{proof}
Suppose $\mathcal{G}=\mathcal{K}\oplus \mathcal{P}$.  By assumption $\mathcal{G}$ is an OSAKA. Hence $\mathcal{K}=Fix(\rho)$ is a loop algebra of the compact type. $\rho^*$ is an involution on $\mathcal{G}^*$. $Fix(\rho^*)=Fix(\rho)=\mathcal{K}$ is hence also a loop algebra of the compact type. Thus $(\mathcal{G}^*, \rho^*)$ is an OSAKA proving (1).\\ If $(\mathcal{G}, \rho)$ is irreducible then so is $(\mathcal{G}^*, \rho^*)$ and vice-versa.
If $(\mathcal{G},\rho)$ is of compact type then $(\mathcal{G}^*,\rho^*)$ is of non-compact type by results of~\cite{Heintze09}, section 6. .
\end{proof}

\begin{corollary}
The OSAKAs of the following types are dual:
\begin{center}\begin{tabular}{ccc}
type I&$\Leftrightarrow$& type III\\
type II&$\Leftrightarrow$&type IV
\end{tabular} 
\end{center}
\end{corollary}

\subsection*{OSAKAs of Euclidean type}

The last class consists of OSAKAs of Euclidean type. Let $\mathfrak{a}$ be an Abelian real Lie algebra. Then the pair $\left(\widehat{L}(\mathfrak{a}), -\Id\right)$ is an OSAKA of Euclidean type. It is irreducible iff $\mathfrak{a}$ is $1$-dimensional.

\section{OSAKAs of type $\mathfrak{a}_{1}^{(1)}$}
\label{sect:a_1_example}

In this final section we calculate as an extended example all OSAKAs associated to the complex affine Kac-Moody algebra of type $\mathfrak{a}_{1}^{(1)}$. Even being the smallest available example, the calculations are quite involved. We start by recalling the real forms and involutions of $\mathfrak{sl}(2, \mathbb{C})$ and describe then the real forms of second type as they are constructed in~\cite{Heintze09}. Then we construct the $8$ different OSAKAs.

\subsection*{Real forms of type \boldmath$A_1^1$}

Recall that $A_1^1$ is the non-twisted affine Kac-Moody algebra $\widehat{L}(\mathfrak{sl}(2, \mathbb{C}))$ associated to the loop algebra 
$$L(\mathfrak{sl}(2, \mathbb{C})):=\{f:S^1\longrightarrow \mathfrak{sl}(2, \mathbb{C})|  f\ \textrm{satisfies some regularity conditions}\}\,.$$ 

Its real forms are described in~\cite{Heintze09} sect.7. We refer to this reference for additional details and proofs. From our classification of OSAKAs, we see that we need for the construction of OSAKAs 

\begin{itemize}
\item the compact real form,
\item the almost split real forms.
\end{itemize}

The compact real form has an appealing description as the Kac-Moody algebra $\widehat{L}(\mathfrak{su}(2))$ associated to $L(\mathfrak{su}(2))$, where 

$$L(\mathfrak{su}(2)):=\{f:S^1\longrightarrow \mathfrak{su}(2)|  f\ \textrm{satisfies some regularity condition}\}\,.$$

The three involutions of second type of $L(\mathfrak{su}(2))$ as the three almost split real forms of $L(\mathfrak{sl}(2, \mathbb{C}))$ are associated to pairs $[\rho_+, \rho_-]$ of involutions of the finite dimensional Lie algebra $\mathfrak{su}(2)$. Here the identity is taken to be an involution.  Two pairs $[\rho_+, \rho_-]$ and $[\widetilde{\rho}_+, \widetilde{\rho}_-]$ define the same involution (and in consequence isomorphic real form) iff 
$$\widetilde{\rho}_+=\alpha \rho_{\pm} \alpha^{-1}\quad \textrm{and}\quad \widetilde{\rho}_-=\beta \rho_{\mp} \beta^{-1}\,,$$
for some automorphisms $\alpha, \beta\in \textrm{Aut}(\mathfrak{g})$ such that $\alpha^{-1}\beta\in \textrm{Int}(\mathfrak{g})$ (see~\cite{Heintze09}, section 5.)

Besides the identity,  $\mathfrak{su}(2)$, has only one involution: complex conjugation, denoted by $\mu$. Hence we get the following $3$ pairs of invariants~\cite{Heintze09}, section 7:

$$[\rho_+, \rho_-]=[\Id,\Id],\quad [\rho_+, \rho_-]=[\Id, \mu],\quad [\rho_+, \rho_-]=[\mu,\mu]\,.$$
The involution associated to a pair of invariants $[\rho_+, \rho_-]$ is defined by $f(t)\mapsto \rho_+(f(-t))$ on the Lie algebra $L(\mathfrak{g}, \sigma)$ such that $\sigma=\rho_-\rho_+$.

The associated real forms can now be described as follows (In the following expressions the power series $\sum_n u_n e^{int}$ are understood to represent elements in $\widehat{L}(\mathfrak{sl}(2, \mathbb{C}))$, hence functions of some prescribed regularity):

\begin{itemize}
\item  $\mathfrak{a}_{1}^{(1)}[\Id, \Id]$ corresponds to series of the form
 $$\left\{\sum u_n e^{int}|u_n\in \mathfrak{su}(2)\right\}\,.$$
\item  $\mathfrak{a}_{1}^{(1)}[\mu, \mu]$ corresponds to series of the form  $$\left\{\sum u_n e^{int}|u_n\in \mathfrak{sl}(2, \mathbb{R})\right\}\,.$$
\item  $\mathfrak{a}_{1}^{(1)}[\Id, \mu]$ corresponds to series of the form
 $$\left\{\sum u_n e^{in\frac{t}{2}}|u_n\in \mathfrak{sl}(2, \mathbb{R}), i^n u_n\in \mathfrak{su}(2)\right\}\,.$$
\end{itemize}

\subsection*{OSAKAs of type \boldmath$I$}

Following our classification, OSAKAs of type $I$ consist of pairs consisting of the compact real form $\widehat{L}(\mathfrak{su}(2))$ together with an involution $\widehat{\rho}$, whose fixed-point algebra is a loop algebra. Hence $\widehat{\rho}(c)=-c$ and $\widehat{\rho}(d)=-d$ and $\widehat{\rho}$ is of second kind. As described above, these involutions correspond to the pairs of invariants

$$[\rho_-, \rho_+]=[\Id,\Id],\quad [\rho_-, \rho_+]=[\Id, \mu],\quad [\rho_-, \rho_+]=[\mu,\mu]\,.$$

\subsubsection*{The compact real OSAKA of type \boldmath$\mathfrak{a}_{1}^{(1)}[\Id, \Id]$}

The Lie algebra is $\mathfrak{g}:=\widehat{L}(\mathfrak{su}(2))$, the involution $\rho$ is defined by 
$$\rho(f(t), r_c, r_d) = (f(-t),-r_c, -r_d) \,.$$

\subsubsection*{The compact real OSAKA of type \boldmath$\mathfrak{a}_{1}^{(1)}[\Id, \mu]$}

The Lie algebra is $\mathfrak{g}:=\widehat{L}(\mathfrak{su}(2))$, the involution $\rho$ is defined by 
$$\rho(f(t), r_c, r_d) = \left(\Ad {\textrm{\small $ \left(\begin{array}{cc}1&0\\0&-1\end{array}\right)$\normalsize}}\overline{f(-t)},-r_c, -r_d \right) \,.$$

\subsubsection*{The compact real OSAKA of type \boldmath$\mathfrak{a}_{1}^{(1)}[\mu, \mu]$}

The Lie algebra is $\mathfrak{g}:=\widehat{L}(\mathfrak{su}(2))$, the involution $\rho$ is defined by 
$$\rho(f(t), r_c, r_d) = (\overline{f(-t)},-r_c, -r_d) \,.$$


\subsection*{The OSAKA of type \boldmath$II$}

Kac-Moody symmetric spaces of type $II$ correspond exactly to compact real forms of Kac-Moody groups. 

The associated OSAKA of type $II$ is constructed as follows: Let first 
$$\mathfrak{g}:=\widehat{L}(\mathfrak{su}(2)\times \mathfrak{su}(2))\,.$$
Hence $\mathfrak{g}$ is indecomposable but reducible as a geometric affine Lie algebra. We describe elements in $\mathfrak{g}$ by quadruples $(f(t),g(t), r_c, r_d)$ where $f(t)$ and $g(t)$ denote elements in $L(\mathfrak{su}(2))$ and $r_c$ resp.~$r_d$ denote the $c$~resp.~$d$-extension. The involution $\rho$ is defined by 
$$\rho(f(t), g(t), r_c, r_d)\ =\ (f^*(t), g^*(t), -r_c, -r_d)\,,$$
where $f^*(t):=-\overline{f(t^{-1})}^t$. It is easy to see that $\textrm{Fix}(\rho):=\left\{(f(t), f^*(t), 0,0), f(t)\in L(\mathfrak{su}(2))\right\}$. 
Hence we get
$$\textrm{Fix}(\rho)\cong \widehat{L}(\mathfrak{su}(2))\,,$$
which is a loop group of the compact type.
Hence, the pair $(\mathfrak{g}, \rho)$ is an OSAKA. As the involution $\rho$ exchanges the two $\widehat{L}(\mathfrak{su}(2))$-subalgebras, it is an irreducible OSAKA.

\subsection*{OSAKAs of type \boldmath$III$}

OSAKAs of type III consist of the three almost-split real forms of $\widehat{L}(\mathfrak{sl}(3, \mathbb{C}))$ together with their Cartan involutions.

\subsubsection*{The almost split real form  \boldmath$\mathfrak{a}_{1}^{(1)}[\Id, \Id]$}

This is the real form associated to the invariants $[\Id, \Id]$. The Lie algebra $\mathcal{G}[\Id, \Id]$ is defined by 
$$\mathfrak{a}_{1}^{(1)}[\Id, \Id]:=\left\{\sum u_n e^{int}|u_n\in \mathfrak{su}(2)\right\}\oplus i\mathbb{R}c\oplus i\mathbb{R}d\,.$$

Let us now investigate the involution $\rho$. $\rho$ is the restriction of the Cartan involution $\omega$ to $\mathcal{G}[\Id, \Id]$. 
The Cartan involution $\omega$ of $\widehat{L}(\mathfrak{sl}(2, \mathbb{C}))$ is defined on the generators $h_i, e_i, f_i, i=1,2$ by $\omega(h_i)=-h_i$, $\omega(e_i)=-f_i$ and $\omega(f_i)=-e_i$.

The Cartan involution of $\mathfrak{a}_{1}^{(1)}[\Id, \Id]$ leads to the Cartan decomposition into the $\pm 1$-eigenspaces:

$$\mathfrak{a}_{1}^{(1)}[\Id, \Id]=\mathcal{K}\oplus \mathcal{P}$$

\noindent where the $+1$-eigenspace $\mathcal{K}$ is given by

$$\mathcal{K}:=\left\{\sum u_n e^{int}|u_n\in \mathfrak{su}(2)|u_n=-\overline{u}_{-n}^t\right\}\,.$$

\noindent and the $-1$-eigenspace $\mathcal{P}$ is given by

$$\left\{\sum u_n e^{int}|u_n\in \mathfrak{su}(2)|u_n=+\overline{u}_{-n}^t\right\}\,.$$

\noindent Remark, that $\mathcal{K}\subset L(\mathfrak{su}(2))$; hence $\mathcal{K}$ is a loop algebra of the compact type and the pair

$$\left(\mathfrak{a}_{1}^{(1)}[\Id, \Id], \rho\equiv\omega\right)$$
is an OSAKA of the non-compact type.


\subsubsection*{The non-compact OSAKA of type \boldmath$\mathfrak{a}_{1}^{(1)}[\Id, \mu]$}

The Lie algebra $\mathcal{G}$ is the real form associated to the invariants $[\Id, \mu]$. Hence the Lie algebra $\mathcal{G}$ is defined by 

$$\mathfrak{a}_{1}^{(1)}[\Id, \mu]:=\left\{\sum u_n e^{in\frac{t}{2}}|u_n\in \mathfrak{sl}(2, \mathbb{R})\ \textrm{and}\ (i)^{n}u_n\in \mathfrak{su}(2)\right\}\oplus i\mathbb{R}c\oplus i\mathbb{R}d\,.$$

Let us now investigate the involution $\rho$. $\rho$ is the restriction of the Cartan involution $\omega$ to $\mathcal{G}[\Id, \Id]$. 
The Cartan involution $\omega$ is defined on the generators $h_i, e_i, f_i, i=1,2$ by $\omega(h_i)=-h_i$, $\omega(e_i)=-f_i$ and $\omega(f_i)=-e_i$.

The Cartan involution of $\mathfrak{a}_{1}^{(1)}[\Id, \Id]$ leads to the Cartan decomposition into the $\pm 1$-eigenspaces:
$$\mathfrak{a}_{1}^{(1)}[\Id, \mu]=\mathcal{K}\oplus \mathcal{P}$$

\noindent where the $+1$-eigenspace $\mathcal{K}$ is given by
$$\mathcal{K}:=\left\{\sum u_n e^{int}|u_n\in \mathfrak{sl}(2, \mathbb{R})\ \textrm{and} (i)^n u_n\in \mathfrak{su}(2)|u_n=-{u}_{-n}^t\right\}\,,$$

\noindent and the $-1$-eigenspace $\mathcal{P}$ is given by
$$\left\{\sum u_n e^{int}|u_n\in \mathfrak{sl}(2, \mathbb{R})\ \textrm{and} (i)^n u_n\in \mathfrak{su}(2)|u_n=+{u}_{-n}^t\right\}\,.$$

\noindent Remark, that $\mathcal{K}\subset L(\mathfrak{su}(2))$; hence $\mathcal{K}$ is a loop algebra of the compact type and the pair
$$\left(\mathfrak{a}_{1}^{(1)}[\Id, \mu], \rho\equiv\omega\right)$$
is an OSAKA of the non-compact type.


\subsubsection*{The non-compact OSAKA of type \boldmath$\mathfrak{a}_{1}^{(1)}[\mu, \mu]$}

From an algebraic point of view, the split real form is defined as the Lie algebra generated by a realization $\mathfrak{h}_{\mathbb{R}}$ and generators $e_1$, $e_2$ and $f_1$, $f_2$ subject to the Cartan-Serre relations. This is possible as a generalized Cartan matrix has only integer coefficients.

The Lie algebra associated to the pair of invariants $[\mu, \mu]$ is seen to be the Lie algebra 
$$\mathfrak{a}_{1}^{(1)}[\mu,\mu]:=\left\{\sum u_n e^{int}|u_n\in \mathfrak{sl}(2,\mathbb{R})\right\}\oplus i\mathbb{R}c\oplus i\mathbb{R}d\,.$$

The Cartan involution of $\mathfrak{a}_{1}^{(1)}[\mu, \mu]$ leads to the Cartan decomposition into the $\pm 1$-eigenspaces:
$$\mathfrak{a}_{1}^{(1)}[\mu, \mu]=\mathcal{K}\oplus \mathcal{P}$$

\noindent where the $+1$-eigenspace $\mathcal{K}$ is given by
$$\mathcal{K}=\left\{\sum u_n e^{int}|u_n\in \mathfrak{sl}(2,\mathbb{R})|u_n=-{u}_{-n}^t\right\}\,.$$

\noindent The $-1$-eigenspace $\mathcal{P}$ of the Cartan involution is defined by.
$$\mathcal{P}=\left\{\sum u_n e^{int}|u_n\in \mathfrak{sl}(2,\mathbb{R})|u_n={u}_{-n}^t\right\}\,.$$

\noindent Clearly $\mathcal{K}\subset L(\mathfrak{su}(2))$ and the pair
$$\left(\mathcal{G}[\mu, \mu], \rho\equiv\omega\right)$$
 is an OSAKA of the non-compact type.


\subsection*{OSAKA of type \boldmath$IV$}

This type diverges a little further from the form, one would expect from finite dimensional theory. The crucial point is, that the pair consisting of $\widehat{L}\left(\mathfrak{sl}(2, \mathbb{C})\right)$ together with its Cartan involution is not an OSAKA in the sense of our definition. 

Let first $\mathcal{G}_0:=L(\mathfrak{sl}(2, \mathbb{C})\otimes \mathfrak{sl}(2, \mathbb{C}))$. We decompose 
$$\mathfrak{sl}(2, \mathbb{C}):=\mathfrak{su}(2)\oplus i \mathfrak{su}(2)$$
and define elements in $\mathcal{G}_0$ via this decomposition as quadruples
$(f(t), g(t), h(t), k(t))$ such that $f(t), h(t) \in L(\mathfrak{su}(2))$ and $g(t), k(t) \in i L(\mathfrak{su}(2))$.

\noindent Define now the subalgebra
$\mathcal{G}_{\varphi}:=\{(f(t), g(t), f(-t), -g(-t))\}$. Then $\mathcal{G}_{\varphi}\cong L(\mathfrak{sl}(2, \mathbb{C}))$.

\noindent Then we define 
$$\mathcal{G}:=\mathcal{G}_{\varphi}\oplus \mathbb{R}c\oplus \mathbb{R}d\,.$$
\noindent and the involution $\rho:\mathcal{G}\longrightarrow \mathcal{G}$ is defined on hexatuples $(f(t), g(t), f(-t), -g(-t), r_c, r_d)$ by 
$$\omega(f(t), g(t), f(-t), -g(-t), r_c, r_d)=(f(t), -g(t), f(-t), g(-t), -r_c, -r_d)\,.$$

\noindent $\rho$ defines a splitting of $\mathcal{G}$ into its $\pm 1$-eigenspaces:
$$\mathcal{G}=\mathcal{K}\oplus \mathcal{P}$$

\noindent where the $+1$-eigenspace $\mathcal{K}$ is given by
$$\mathcal{K}:=\left\{f(t), 0, f(-t), 0, 0, 0\right\}\,.$$

\noindent and the $-1$-eigenspace $\mathcal{P}$ is given by
$$\mathcal{P}=\left\{0, g(t), 0, -g(-t), r_c, r_d\right\}\,.$$

\noindent Remark, that $\mathcal{K} \cong L(\mathfrak{su}(2))$; hence $\mathcal{K}$ is a loop algebra of the compact type and the pair
$$(\mathcal{G}, \rho)$$
is an OSAKA of the non-compact type. Furthermore $\rho$ is the Cartan involution on $\mathcal{G}$.


\bibliographystyle{alpha}
\bibliography{Doktorarbeit1}

\end{document}